\documentclass[preprint,12pt]{elsarticle}

\usepackage{amssymb}
\usepackage{amsmath}

\journal{Nuclear Physics B}
\usepackage{amsmath,amssymb,amscd}
\allowdisplaybreaks[4]
\usepackage{latexsym}
\usepackage{graphicx}
\usepackage{subcaption}
\usepackage{caption}
\usepackage{amssymb}
\usepackage{autobreak}
\usepackage{amsthm,amsmath,amssymb}
\usepackage{mathrsfs}
\usepackage{xcolor}
\usepackage{geometry}
\usepackage[pdftex]{hyperref}
\makeatletter \@addtoreset{equation}{section} \makeatother

\newtheorem{theorem}{Theorem}[section]
\newtheorem{lemma}[theorem]{Lemma}

\newtheorem{corollary}[theorem]{Corollary}

\theoremstyle{definition}
\newtheorem{example}[theorem]{Example}

\theoremstyle{plain}
\newtheorem{remark}[theorem]{Remark}
\newtheorem{assumption}[theorem]{Assumption}
\begin{document}

\begin{frontmatter}



\title{Exponential ergodicity of exact and numerical solutions for McKean-Vlasov SDEs driven by L{\'e}vy noise} 

\author[mymainaddress]{Yang Sun}
\author[mysecondaryaddress1]{Yuhang Zhang}
\author[mymainaddress]{Minghui Song\corref{mycorrespondingauthor}}
\cortext[mycorrespondingauthor]{Corresponding author}
\ead{songmh@hit.edu.cn}
\address[mymainaddress]{School of Mathematics,
	Harbin Institute of Technology, Harbin, China, 150001.}
\address[mysecondaryaddress1]{School of Science, Harbin Institute of Technology, Shenzhen, China, 518055.} 


\begin{abstract}
 
 This paper investigates the exponential ergodicity of the exact solution and the tamed Euler solution for McKean--Vlasov stochastic differential equations driven by L\'evy noise.  First, we  establish exponential ergodicity for both the original equation and the tamed Euler method. Then we prove the convergence of the numerical invariant measure to the exact invariant measure, which is obtained  by combining the propagation of chaos (PoC) result with the strong convergence of the tamed Euler scheme. Furthermore, we derive a convergence rate for the numerical invariant measure by establishing uniform-in-time PoC and uniform-in-time convergence of the tamed Euler method. Finally, numerical experiments are presented to illustrate the theoretical results.
\end{abstract}




\end{frontmatter}



\section{Introduction}
\label{sec1}
 McKean–Vlasov stochastic differential equations (MV-SDEs) are a class of stochastic differential equations whose coefficients depend not only on the current state of the solution but also on its probability distribution. These equations were initially introduced by McKean \cite{McKean} in the 1960s and are now widely applied in areas such as finance, physics, biology, neuroscience and machine learning, see \cite{Chiara,Kolokolnikov,Sabbar2025,NEUFELD2025,Baladron,Bossy2015,Ullner2018}, as well as the  monographs \cite{Bensoussan,CarmonaI2018,CarmonaII2018}. In general,   MV-SDEs are difficult to solve explicitly, as the law of the solution is generally unknown. Consequently, numerical methods have attracted considerable research interest. A standard approach to approximating MV-SDEs is to introduce an associated interacting particle system. This is motivated by the fact that MV-SDEs describe the limiting mean-field behavior of interacting particles, a phenomenon known as propagation of chaos (PoC). Combining the PoC with time discretization provides the general framework for the numerical method of MV-SDEs. Over the past few decades, extensive studies have been devoted to numerical approximation of MV-SDEs, see \cite{Li2023,BaoReis2021,Reis2022,Ding2021,Kumar2021} and the references therein. 

 In many practical applications, systems may be subject to sudden shocks and discontinuities, such as price jumps in financial markets and abrupt spike events in neural firing (see, e.g., \cite{Amini,Andreis2018,Sima2020} and references therein), which cannot be captured by Gaussian noise alone, thus L{\'e}vy noise is incorporated to model such phenomena.  Numerical methods for MV-SDEs driven by L\'evy noise have therefore attracted increasing attention. For example, in \cite{Agarwal2020}, a Fourier-based numerical approximation was developed for MV-SDEs with L{\'e}vy noise, where the noise has no law dependency and the drift contains a linear mean-field interaction term. In \cite{Neelima2010}, a tamed Euler scheme was introduced under monotonicity and coercivity conditions, allowing the drift, diffusion and jump coefficients to have superlinear growth in the state variable. In \cite{Sani}, a  tamed Milstein scheme was proposed, where a strong \(L^2\)-convergence rate close to one was obtained. In \cite{Ngo2025}, a tamed-adaptive Euler method was studied for L{\'e}vy-driven MV-SDEs with superlinear drift and diffusion coefficients and a linear jump coefficient, and strong convergence was established on both finite and infinite time intervals. More recently, \cite{Zhu2025} investigated a general class of Euler-type schemes for L\'evy-driven MV-SDEs with superlinear coefficients and obtained a mean-square convergence rate close to \(1/2\) without imposing the coercivity condition used in \cite{Neelima2010}. 

 The existence and uniqueness of invariant measures, together with ergodicity, provide a mathematical characterization of the long-time statistical equilibrium of stochastic systems. For MV-SDEs, the law dependence of the coefficients makes such long-time behavior more involved, and this topic has attracted considerable attention in recent years. For instance, exponential ergodicity in the \(L^2\)-Wasserstein distance was established in \cite{WANG2018} under a uniform dissipativity condition. Ergodicity under dissipativity assumptions in long distances was investigated in \cite{BAO2025} and \cite{BAO202512}. For MV-SDEs with fully superlinear coefficients, exponential ergodicity and the existence of invariant measures were studied in \cite{2025dos}. The exponential ergodicity of pure-jump MV-SDEs was considered in \cite{Wang2021}, while the existence, multiplicity and uniqueness of stationary distributions for MV-SDEs with jumps were examined in \cite{Wang2026}.  Moreover, the existence and approximation of measure attractors and invariant measures for McKean-Vlasov stochastic lattice system with L{\'e}vy noise were investigated in \cite{BAI2026}.
 
 Although the long-time dynamical behavior of exact solutions to MV-SDEs has been extensively studied, explicit forms of invariant measures are generally unavailable. This makes it essential to develop reliable numerical methods for approximating long-time dynamics, while such results remain relatively limited. Recently, \cite{soni2025} established exponential ergodicity and long-time approximation results for MV-SDEs, their interacting particle systems, and the corresponding tamed Euler scheme under fully superlinear coefficients. In \cite{Cui2024}, a truncated Euler method was developed for MV-SDEs with superlinear growth coefficients; the existence and uniqueness of the numerical invariant measure were proved, together with the convergence between the numerical and exact invariant measures. However, to the best of our knowledge, long-time approximation results for MV-SDEs driven by L\'evy noise have not yet been established. Therefore, in this paper, we consider the exponential ergodicity of the MV-SDE and whether numerical methods can preserve its ergodic property.

 The remainder of this paper is organized as follows. 
 Section~2 presents the notation and preliminaries. 
 Sections~3 and~4 establish the exponential ergodicity of the exact MV-SDE and the tamed Euler scheme, respectively. 
Section~5 establishes finite-time PoC and time-discretization estimates, leading to the convergence of the numerical invariant measure.
Section~6 establishes uniform-in-time PoC and time-discretization estimates, leading to a convergence rate for the numerical invariant measure. 
 Section~7 provides numerical experiments to support the theoretical results.
 \section{Notations and preliminaries }
 \subsection{Notations}
 Let $ (\Omega,\mathcal F,\{{\mathcal F}_t\}_{t\geq 0}, \mathbb{P}) $ be a complete probability space for which the filtration $ \{{\mathcal F}_t\}_{t\geq 0} $ is assumed to satisfy the usual conditions, i.e., it is right continuous and $ {\mathcal F}_0 $ contains all $ \mathbb{P} $-null sets. 

  Throughout this paper, for a vector $ x\in \mathbb{R}^d  $, $ \left| x\right|  $ denotes the Euclidean norm. For a matrix $  A \in \mathbb{R}^{d\times m}$, $ A^\mathrm{T} $ denotes its transpose, and
 $ \left| A\right|=\sqrt{\mathrm{trace}(A^\mathrm{T}A)} $ denotes its Frobenius norm. The inner product of vectors $ x$ and $ y $ is written as $ \left\langle x,y\right\rangle  $.  The Dirac measure at point $ x $ is denoted by $ \delta_x $. For any $a,b\in \mathbb{R}$, let $a\lor b=\max\left\lbrace a,b\right\rbrace $ and $a\land b=\min\left\lbrace a,b \right\rbrace $.  The set of positive integers is denoted by $\mathbb{N}$ and $\mathbb{N}_0:=\mathbb{N}\cup \left\lbrace0 \right\rbrace $.
 
  Now we introduce the set of all probability measures on $ (\mathbb{R}^d,\mathcal B(\mathbb{R}^d)) $, denoted by $ \mathcal P(\mathbb{R}^d) $, where $ \mathcal B(\mathbb{R}^d) $ is the Borel $ \sigma $-field over $ \mathbb{R}^d $. The subspace of $ \mathcal P(\mathbb{R}^d) $ with finite $ r $-th moment is defined as
 \begin{equation*}
 	\mathcal P_r(\mathbb{R}^d):=\left\lbrace \mu\in \mathcal P(\mathbb{R}^d)\middle|\int_{\mathbb{R}^d}\left| x\right|^r\mu (\mathrm{d}x)<\infty  \right\rbrace .
 \end{equation*}
 Then for $ r\geq 1 $, $ \mathcal{P}_r(\mathbb{R}^d) $ is a Polish space endowed with the Wasserstein distance
 \begin{equation*}
 	\mathbb{W}_r(\mu ,\sigma ):=\inf_{\pi \in \mathcal{C}
 		(\mu ,\sigma )}\left( \int_{\mathbb{R}^d\times\mathbb{R}^d}\left| x-y\right|^r\pi (\mathrm{d}x,\mathrm{d}y) \right) ^{\frac{1}{r}}, ~  \mu ,\sigma \in \mathcal P_r(\mathbb{R}^d),
 \end{equation*}
 where $\mathcal{C}(\mu ,\sigma )$ is the set of couplings for $ \mu $ and $ \sigma $. In other words, $ \pi \in \mathcal{C}(\mu ,\sigma ) $  is a probability measure such that $ \pi (\cdot ,\mathbb{R}^d)=\mu(\cdot) $ and $ \pi (\mathbb{R}^d,\cdot)=\sigma(\cdot) $. 
 \par Finally, we note that $ C>0 $ denotes a generic constant, which is always independent of the numerical step size $ \Delta $ and the number of particles $ N $, and its value may vary from one instance to another.
 \subsection{Preliminaries}
 In this article, we consider the following  MV-SDEs driven by L{\'e}vy noise
 \begin{align}\label{eq1.1}
 	\mathrm{d}x(t)=&b(x(t),\mu_t^x)\mathrm{d}t+g(x(t),\mu_t^x)\mathrm{d}B(t)+\int_{Z}\gamma(x(t),\mu_t^x,z)\widetilde{N}(\mathrm{d}t,\mathrm{d}z)
 \end{align}
 with initial value $x(0)=\xi$, where  $ b:\mathbb{R}^d\times \mathcal{P}_2(\mathbb{R}^d)\rightarrow\mathbb{ R}^d, g:\mathbb{R}^d\times \mathcal{P}_2(\mathbb{R}^d)\rightarrow\mathbb{ R}^{d\times m} , \gamma: \mathbb{R}^d\times \mathcal{P}_2(\mathbb{R}^d)\times Z\rightarrow\mathbb{ R}^d $ are Borel measurable maps, $ \mu^x_t $ denotes the law of $ x(t) $, $ B(t) $ is an $m$-dimensional standard Brownian motion, $ N(\mathrm{d}t,\mathrm{d}z) $ is the Poisson random measure on a $ \sigma $-finite measure space $ (Z,\Sigma,\nu ) $ with $ Z \subseteq \mathbb{R}^d\backslash\{0\} $, independent of $ B(t) $, and $ \widetilde{N}(\mathrm{d}t,\mathrm{d}z):= N(\mathrm{d}t,\mathrm{d}z)-\nu(\mathrm{d}z)\mathrm{d}t $ is its compensated Poisson random measure, $ \nu $ denotes the intensity measure satisfying $ \nu(Z)<\infty $. 
 
 We now proceed to state the assumptions.
  \begin{assumption}\label{2.1}
 $	\mathbb{E}\left| \xi \right|^{2} <\infty$. 
 \end{assumption}
  \begin{assumption}\label{dissipative}
 	There exist constants $\lambda_{1},\lambda_{2}> 0$ such that for any $ x_1,x_2\in \mathbb{R}^d $ and $
 	\mu_1,\mu_2 \in \mathcal P_2(\mathbb{R}^d)  $,
 	\begin{eqnarray*}
 		\begin{split}
 			\left\langle x_1-x_2,b(x_1,\mu_1)-b(x_2,\mu_2)\right\rangle
 			\leq -\lambda_{1}\left|x_1-x_2 \right|^2+\lambda_{2}\mathbb{W}^2_2(\mu_1,\mu_2).
 		\end{split}
 	\end{eqnarray*}
 \end{assumption}
 \begin{assumption}\label{g}
 	There exists a constant $\lambda_{3}> 0$ such that for any $ x_1,x_2\in \mathbb{R}^d $ and $
 	\mu_1,\mu_2 \in \mathcal P_2(\mathbb{R}^d)  $,
 	\begin{eqnarray*}
 		\begin{split}
 		\left|g(x_1,\mu_1)-g(x_2,\mu_2)\right|^2\vee\int_{Z}\left| \gamma(x_1,\mu_1,z)-\gamma(x_2,\mu_2,z)\right|^{2} \nu(\mathrm{d}z)	\leq \lambda_{3}\big( \left|  x_1-x_2 \right|^2 +\mathbb{W}^2_2(\mu_1,\mu_2)\big).
 		\end{split}
 	\end{eqnarray*}
 \end{assumption}
 \begin{remark} \label{re1}
 	It follows directly from Assumptions \ref{dissipative} and \ref{g} that  for any $x \in \mathbb{R}^{d}$ and $\mu \in \mathcal{P}_{2}(\mathbb{R}^{d})$,
 	\begin{eqnarray}\label{3.16}
 		\left\langle x,b(x,\mu)\right\rangle \leq -\left( \lambda_{1}-\frac{1}{2}\right) \left|x\right|^2+\lambda_{2}\mathbb{W}^2_2(\mu,\delta_0)+C,
 	\end{eqnarray}
 	where $C=\frac{1}{2} \left|b(0,\delta_0) \right|^2 $, and
 	\begin{align}\label{5.2}
 		\left| g(x,\mu)\right|^2\vee\int_{Z}\left| \gamma(x,\mu,z)\right|^{2} \nu(\mathrm{d}z)
 		\leq 2\lambda_{3}\big( \left|  x\right|^2 +\mathbb{W}^2_2(\mu,\delta_{0})\big)\!+C,
 	\end{align}
 	where  $C=2 \left|g(0,\delta_0) \right|^2\vee2\int_{Z}\left|\gamma(0,\delta_0,z) \right|^2\nu(\mathrm{d}z)$.
 \end{remark}

 According to \cite{Neelima2010}, we know that Assumptions \ref{2.1}-\ref{g}  imply the existence of a unique strong solution to the MV-SDE \eqref{eq1.1}.
 
  \section{Exponential ergodicity of the MV-SDE}
 In this section, we study the ergodic property of MV-SDE \eqref{eq1.1}. Before giving the main results,
 we define an operator $P_{t}  $ on $\mathcal{P}_{2}(\mathbb{R}^d)$ by letting $P_{t}\mu :=\mathcal{L}(x(t))$ for $\mathcal{L}(x(0))=\mathcal{L}(\xi)=\mu$. According to \cite{WANG2018}, $\left\lbrace P_{t} \right\rbrace_{t \geq 0} $ satisfies $P_{s+t} = P_{s}P_{t}$ for any $s, t \geq 0$. 

 We begin this section with the following lemma.
 \begin{lemma}\label{infty bound}
 	Let Assumptions \ref{2.1}-\ref{g} hold with $2\lambda_1-2\lambda_2-8\lambda_{3}-1>0$. Then
 	\begin{align}
 		\sup_{t\geq 0}\mathbb{E}| x(t)|^2\leq C.\notag
 	\end{align}
 \end{lemma}
 \begin{proof}
 	Using Itô's formula, it follows from \eqref{3.16} and \eqref{5.2} that for any $\lambda > 0$, 
 	\begin{align*}
 		&\mathbb{E} \left( e^{\lambda t} |x(t)|^2 \right) \\
 		\leq & \mathbb{E}|x(0)|^2 +\lambda \mathbb{E}\left(  \int_0^t e^{\lambda s} |x(s)|^2 \, \mathrm{d}s \right)  \\
 		&+ \! \mathbb{E} \Bigg[ \int_0^t e^{\lambda s}\bigg( 2 \langle x(s), b( x(s), \mu_s^x) \rangle +\left| g(x(s),\mu_s^x) \right|^2  +\int_{Z}\left| \gamma(x(s),\mu_s^x,z)\right| ^2\nu(\mathrm{d}z)\,\bigg)  \mathrm{d}s \Bigg]  \\
 		\leq & \mathbb{E} |\xi|^2
 		+ \left(  \lambda -2 \lambda_1+2\lambda_2+8\lambda_{3}+1\right)  \int_0^t e^{\lambda s} \mathbb{E}|x(s)|^2 \, \mathrm{d}s +C \int_0^t e^{\lambda s} \mathrm{d}s,
 	\end{align*}
 	where we use $\mathbb{E}\mathbb{W}_2^2(\mu_s^x,\delta_0)\leq\mathbb{E}|x(s)|^2 $. By setting $\lambda=2\lambda_1-2\lambda_2-8\lambda_{3}-1>0$, we obtain
 	\begin{align*}
 		\mathbb{E}|x(t)|^2 
 		\leq& e^{-\lambda t} \mathbb{E} |\xi|^2 +\frac{C}{\lambda}
 	\end{align*}
 	for all $t \geq 0$, which completes the proof.
 \end{proof}
 
 \begin{theorem}\label{4.2}
 	Let Assumptions \ref{2.1}-\ref{g} hold with $2\lambda_1-2\lambda_2-8\lambda_{3}-1>0$. Then, there exists a positive constant $\lambda^*$, independent of time $t$, such that
 	\begin{enumerate}
 		\item[(1)] for any $\mu$, $\sigma\in\mathcal{P}_{2}(\mathbb{R}^d)$,
 		\begin{align}\label{**}
 			\mathbb{W}_{2}^{2}(P_{t}\mu,\,P_{t}\sigma)\leq  e^{-\lambda^* t}\mathbb{W}_{2}^{2}(\mu,\,\sigma),\ \ t\geq 0.
 		\end{align}
 		\item[(2)] MV-SDE \eqref{eq1.1}  has a unique invariant probability measure $\mu^*\in\mathcal{P}_{2}(\mathbb{R}^d)$, and
 		\begin{align*}
 			\mathbb{W}_{2}^{2}(P_{t}\mu,\mu^*)\leq e^{ -\lambda^* t}\mathbb{W}_{2}^{2}(\mu,\mu^*),\ \ t\geq 0,\,\mu\in\mathcal{P}_{2}(\mathbb{R}^d).
 		\end{align*}
 	\end{enumerate}
 \end{theorem}
 \begin{proof}
 	(1) Let  $x(t)$ and $y(t)$ be two solutions to MV-SDE \eqref{eq1.1}
 	such that $\mathcal{L}(x(0))=\mu$, $\mathcal{L}(y(0))=\sigma$ and
 	\begin{align}\label{eq4.1}
 		\mathbb{W}_{2}^{2}(\mu,\,\sigma)=\mathbb{E}| x(0)-y(0)|^{2}.
 	\end{align}
 	
 	 For any $\lambda>0$, using   Itô's formula, Assumptions \ref{dissipative} and \ref{g}, one has 
 	\begin{align*}
 		&\mathbb{E}\left( 	e^{\lambda t}|x(t)-y(t)|^{2}\right) \\
 		\leq& \mathbb{E}|x(0)-y(0)|^{2}
 		+\!\mathbb{E}\bigg[ \int_{0}^{t}\! e^{\lambda s}\left((\lambda-\!2\lambda_{1}+2\lambda_3) |x(s)-\!y(s)|^{2}+2(\lambda_{2}+\lambda_3)\mathbb{W}^{2}_{2}(\mu_s^x,\mu_s^y)\right) \mathrm{d}s\bigg] \\
 		\leq &\mathbb{E}|x(0)-y(0)|^{2}\!+\!\left(\lambda-2\lambda_{1}+2\lambda_{2}+4\lambda_3\right) \int_{0}^{t}\! e^{\lambda s}\mathbb{E} |x(s)-\!y(s)|^{2} \mathrm{d}s,
 	\end{align*}
 	where we also use $\mathbb{E}\mathbb{W}_{2}^{2}(\mu_s^x,\mu_s^y)\leq\mathbb{E}|x(s)-y(s)|^{2}$. Since $2\lambda_1-2\lambda_2-8\lambda_{3}-1>0$, we can choose $\lambda^* > 0$ such that $\lambda^*-2\lambda_{1} +2\lambda_{2}+4\lambda_3=0$. Take $\lambda=\lambda^*$, then for all $t \geq 0$,
 	\begin{align*}
 		\mathbb{W}_{2}^{2}(P_{t}\mu,\,P_{t}\sigma)\leq\mathbb{E}|x(t)-y(t)|^{2}\leq e^{-\lambda^* t}\mathbb{E}|x(0)-y(0)| ^{2}.
 	\end{align*}
 	Hence, the assertion follows from this inequality and \eqref{eq4.1}.
 	
 	(2) For any $\mu\in\mathcal{P}_{2}(\mathbb{R}^d) $, we first prove
 	\begin{align}\label{eq4.2}
 		\lim_{t\to\infty}\mathbb{W}_{2}(P_{t}\mu,\,\mu^*)=0
 	\end{align}
 	for some $\mu^*\in\mathcal{P}_{2}(\mathbb{R}^d)$. It suffices to show that $\left\lbrace P_{t}\mu\right\rbrace_{t\geq 0} $ is a $\mathbb{W}_{2} $-Cauchy family as $t\to\infty$, that is
 	\begin{align}\label{eq4.3}
 		\lim_{t\to\infty}\sup_{s\geq 0}\mathbb{W}_{2}(P_{t}\mu,\,P_{t+s}\mu)=0.
 	\end{align}
 	\par 	Let $x(0)\sim\mu$ and note that, for any $t,s\geq 0$, we have $P_{t+s}\mu=P_{t}(P_{s}\mu)$. Using \eqref{**},	 we obtain
 	\begin{align*}
 		\mathbb{W}^{2}_{2}(P_{t+s}\mu,\,P_{t}\mu)\leq& \mathrm{e}^{-\lambda^* t}\mathbb{W}^{2}_{2}(P_{s}\mu,\mu)\\
 		\leq&\mathrm{e}^{-\lambda^* t}\mathbb{E}| x(s)-x(0)| ^{2}\\
 		\leq & 2\mathrm{e}^{-\lambda^* t}\left(\mathbb{E}| x(s)|^{2} +\mathbb{E}| x(0)| ^{2}\right).
 	\end{align*}
 	Combining Lemma \ref{infty bound}, \eqref{eq4.3} holds. Because \(\mathcal{P}_{2}(\mathbb{R}^d)\) is complete under the metric $\mathbb{W}_2$, \eqref{eq4.2} holds.
 	Moreover, by \eqref{**} and \eqref{eq4.2}, we have
 	\[
 	\lim_{t \to \infty} \mathbb{W}_2( P_sP_t\mu,P_s\mu^*) = 0, \quad s \geq 0.
 	\]
 	Combining this with \eqref{eq4.2} and \eqref{eq4.3}, we obtain by triangle inequality that for any $s\geq 0,$ 
 	\[
 	\mathbb{W}_2(P_s\mu^*, \mu^*) \leq \lim_{t \to \infty} \mathbb{W}_2(P_sP_t\mu, P_t\mu) = \lim_{t \to \infty} \mathbb{W}_2(P_{t+s}\mu, P_t\mu) = 0.
 	\]
 	Then \(\mu^*\) is an invariant probability measure. Therefore from \eqref{**}, one can see that for \(\mu\in \mathcal{P}_2(\mathbb{R}^d)\)
 	\begin{align*}
 		\mathbb{W}_2^2(P_t\mu, \mu^*) =\mathbb{W}_2^2(P_t\mu, P_t\mu^*) \leq e^{-\lambda^*t} \mathbb{W}^2_2(\mu, \mu^*), \quad t \geq 0.
 	\end{align*}
 	This implies the exponential ergodic property of MV-SDE \eqref{eq1.1}.
 \end{proof}
 \section{Exponential ergodicity of the tamed Euler method}
 \par In this section, we study the numerical ergodicity of the tamed Euler method. For this purpose, we narrowly focus on the equation \eqref{2.1} with the drift term given by
 \begin{align}\label{b1b2}
 	b(x,\mu)=b_1(x)+b_2(x,\mu).
 \end{align}
 Suppose $b_1$ and $b_2$ satisfy the following assumption.
 \begin{assumption}\label{dos}
 	There exist constants $q,L_0, L_1,L_2,L_3,L_4,L_5>0$ such that for any $ x_1,x_2\in \mathbb{R}^d $ and $
 	\mu_1,\mu_2 \in \mathcal P_2(\mathbb{R}^d)  $,
 	\begin{align*}
 		\left|b_1(x_1)-b_1(x_2) \right| ^2\leq& L_0\left(1+\left|x_1 \right| ^{2q}+\left|x_2 \right| ^{2q} \right)\left| x_1-x_2\right|^2,\\
 		\left\langle x_1-x_2,b_1(x_1)-b_1(x_2) \right\rangle
 		\leq& -L_1\left(1+\left|x_1 \right| ^q+\left|x_2 \right| ^q \right)\left| x_1-x_2\right|^2,  \\
 		\left\langle x_1-\!x_2,b_1(x_1)\left|x_2 \right| ^q \!-b_1(x_2)\left|x_1 \right| ^q \right\rangle 
 		\leq &\left(  L_2\left(1+\!\left|x_1 \right| ^q+\!\left|x_2 \right| ^q \right)\!-L_3\left|x_1 \right| ^q\left|x_2 \right| ^q \right)  \left| x_1-\!x_2\right|^2,\\
 		\left|b_1(x_1)\left|x_2 \right| ^q \!-b_1(x_2)\left|x_1 \right| ^q\right| ^2
 		\leq& L_4\left(1+\!\left|x_1 \right| ^{2q}+\!\left|x_2 \right| ^{2q}+\!\left|x_1 \right| ^{2q}\left|x_2 \right| ^{2q} \right)\left| x_1-\!x_2\right|^2,\\
 		\left| b_2(x_1,\mu_1)-b_2(x_2,\mu_2)\right|^2 \leq&L_5\left( \left| x_1-x_2\right|^2+\mathbb{W}_2^2(\mu_1,\mu_2)\right) .
 	\end{align*}
 \end{assumption}
 \begin{remark}
 \rm The function $b_1(x)=-|x|^2x+x$ (with $x\in\mathbb{R}^d $) provides a concrete example satisfying all the conditions on  $b_1$ in Assumption \ref{dos}. For a detailed verification, see \cite{BAO2024}.
 \end{remark}
 
 \begin{remark}\label{remark}
 	\rm	It follows from Assumption \ref{dos} that
 	\begin{align*}
 		&\left\langle x_1-x_2,b(x_1,\mu_1)-b(x_2,\mu_2)\right\rangle \\
 		=&\left\langle x_1-x_2,b_1(x_1)-b_1(x_2)\right\rangle+\left\langle x_1-x_2,b_2(x_1,\mu_1)-b_2(x_2,\mu_2)\right\rangle\\
 		\leq &-L_{1}\left(1+\left|x_1 \right| ^q+\left|x_2 \right| ^q \right)\left|x_1-x_2 \right|^2+\frac{1}{2} \left|x_1-x_2 \right|^2\!+\frac{1}{2} L_5\left(\left| x_1-x_2\right|^2+\mathbb{W}^2_2(\mu_1,\mu_2)\right)\\
 		\leq &-\left( L_{1}-\frac{1}{2}-\frac{1}{2}L_5\right) \left|x_1-x_2 \right|^2+\frac{1}{2} L_5\mathbb{W}^2_2(\mu_1,\mu_2)
 	\end{align*} 
 	and
 	\begin{align*}
 		\left| b(x_1,\mu_1)-b(x_2,\mu_2)\right| ^2\leq 2(L_0\vee L_5)\left(1+\left|x_1 \right| ^{2q}+\left|x_2 \right| ^{2q} \right)\left| x_1-x_2\right|^2+2L_5\mathbb{W}^2_2(\mu_1,\mu_2).
 	\end{align*}
 	Hence, Assumption \ref{dissipative} is satisfied, and Theorem \ref{4.2} holds under Assumptions \ref{2.1}, \ref{dos} and \ref{g} with $L_1-L_5-4\lambda_3-1>0 $.
 \end{remark}

Let $ N\in \mathbb{N} $. For $i=1,2,\cdots,N$, let  $ (\xi^i,B^i,\widetilde{N}^i)$ be independent copies of $ (\xi,B,\widetilde{N}) $ and all $ (\xi^i,B^i,\widetilde{N}^i) $ are mutually independent and identically distributed (i.i.d). The corresponding interacting stochastic particle system (IPS) is
\begin{align}\label{eq2.7}
	\mathrm{d}x^{i,N}(t)=&b(x^{i,N}(t),\mu_t^{x,N})\mathrm{d}t+g(x^{i,N}(t),\mu_t^{x,N})\mathrm{d}B^i(t)+\int_{Z}\gamma(x^{i,N}(t),\mu_t^{x,N},z)\widetilde{N}^i(\mathrm{d}t,\mathrm{d}z)
\end{align} 
with initial value $x^{i,N}(0)=\xi^i$, where $ \mu_t^{x,N}:=\frac{1}{N}\sum_{j=1}^{N}\delta_{x^{j,N}(t)}$.   And the non-interacting particle system (NIPS) is
\begin{align}\label{eq2.8}
	\mathrm{d}x^i(t)=&b(x^i(t),\mu_t^{x^i})\mathrm{d}t+g(x^i(t),\mu_t^{x^i})\mathrm{d}B^i(t)+\int_{Z}\gamma(x^i(t),\mu_t^{x^i},z)\widetilde{N}^i(\mathrm{d}t,\mathrm{d}z)
\end{align}
with initial value $x^{i}(0)=\xi^i$. Since $x^i(t)$ are independent, $ \mu_t^{x^i}=\mu_t^{x} $ for every $i=1,2,\cdots,N $. Under Assumptions \ref{2.1},\ref{g} and \ref{dos}, the existence and uniqueness of the solution to the IPS \eqref{eq2.7} can be established (see \cite{Gyongy1980}). 

  Now, we are devoted to solving the IPS \eqref{eq2.7} numerically. Let $\Delta\in(0,1)$ be the stepsize and set $t_k=k\Delta$ for $k\in \mathbb{N}_0$. The discrete tamed Euler scheme is defined as follows:
 \begin{align}\label{discrete}
 	X^{i,N}(t_{k+1})=&X^{i,N}(t_{k})+ b_\Delta(X^{i,N}(t_{k}),\mu_{t_{k}}^{X,N})\Delta+ g(X^{i,N}(t_{k}),\mu_{t_{k}}^{X,N})\Delta B^i_k\notag\\
 	&+\int_{t_{k}}^{t_{k+1}}\int_{Z} \gamma(X^{i,N}(t_{k}),\mu_{t_{k}}^{X,N},z)\widetilde{N}^i(\mathrm{d}s,\mathrm{d}z)
 \end{align}
 almost surely for all $i\in\left\lbrace1,\cdots,N \right\rbrace $ and $k\in \mathbb{N}_0$, where $X^{i,N}(0)=\xi^i$ and $\Delta B^i_k=B^i(t_{k+1})-B^i(t_k) $. $b_\Delta$ is defined by 
 \begin{align}\label{newtamefuction}
 	b_\Delta(x,\mu):=	\frac{b_1(x)}{1+\sqrt{\Delta}\left|x \right| ^q}+b_2(x,\mu),
 \end{align}
 where $q$ is defined in Assumption \ref{dos}.
 
 For  $\mu\in\mathcal{P}_2(\mathbb{R}^d) $, define $P_{t_k}^{\Delta,N}\mu :=\mathcal{L}(X^{i,N}(t_k)) $, $i=1,\cdots, N$, where $\{X^{i,N}(0)\}_{i=1}^N$ are i.i.d. random variables with same  law $\mu$.
 For the tamed Euler scheme \eqref{discrete} associated with the IPS, we consider the invariant distribution 
 on \((\mathbb{R}^d)^N\) and introduce the joint law as
 \[
 \mathbf{P}_{t_k}^{\Delta,N} \mu^{\otimes N} 
 := \mathcal{L}\big( \left( X^{1,N}(t_k),\dots,X^{N,N}(t_k)\right)  \big),
 \]
 where  \(\mu^{\otimes N}\) is the \(N\)-tensorised initial distribution \(\mu\) and \(\mathcal{L}\big( \left( X^{1,N}(0),\dots,X^{N,N}(0)\right)  \big)=\mu^{\otimes N}\).
 
 In order to show the existence and uniqueness of the numerical invariant measure, we prepare some lemmas on the
 moment boundedness and exponential contractivity.
 \begin{lemma}\label{new1}
 	Let Assumptions \ref{2.1},\ref{g} and \ref{dos} hold, and suppose further $\rho:= L_1-1-6L_5-2\left( m+1\right) \lambda_{3}>0. $ Then,
 	 for any $\Delta\leq 1 \land \left(\frac{\rho}{4L_0} \right) ^2\land \frac{1}{\rho}$,
 	\begin{align*}
 		\sup_{k\in \mathbb{N}_0}\mathbb{E}\left| X^{i,N}_{t_k}\right| ^2\leq C.
 	\end{align*}
 \end{lemma}
 \begin{proof}
 	From scheme \eqref{discrete}, it can be seen that
 	 \begin{align}\label{*}
 		\left| X^{i,N}(t_{k+1}) \right|^2 
 		= 	&\left| X^{i,N}(t_{k})\right|^2 + \bigg( 2\left\langle X^{i,N}(t_{k}), b_{\Delta}\left( X^{i,N}(t_{k}),\mu_{t_k}^{X,N}\right)  \right\rangle \notag\\
 		&+ \!\left| b_{\Delta}\left( X^{i,N}(t_{k}),\mu_{t_k}^{X,N}\right) \right| ^2 \Delta\bigg) \Delta\! + \!2 \left\langle X^{i,N} (t_{k})+\! b_{\Delta}\left( X^{i,N} (t_{k}),\mu_{t_k}^{X,N}\right)\!  \Delta,\right.\notag\\
 		&\left. g\left( X^{i,N}(t_{k}),\mu_{t_k}^{X,N}\right)  \Delta B^i_k+\int_{t_{k}}^{t_{k+1}}\!\int_{Z}\! \gamma\left( X^{i,N}(t_{k}),\mu_{t_k}^{X,N},z\right) \widetilde{N}^i(\mathrm{d}s,\mathrm{d}z)\right\rangle\notag\\
 		&+\!\left| g\left( X^{i,N}(t_{k}),\mu_{t_k}^{X,N}\right) \Delta B^i_k+\!\int_{t_{k}}^{t_{k+1}}\int_{Z} \gamma\left( X^{i,N}(t_{k}),\mu_{t_k}^{X,N},z\right) \widetilde{N}^i(\mathrm{d}s,\mathrm{d}z)\right|^2\notag\\
 		=:& \left| X^{i,N} (t_{k}) \right|^2 + I_1\Delta+I_2.
 	\end{align}
 	By using Assumption \ref{dos} and the Cauchy-Schwarz inequality, we derive that for any $\Delta\leq \left(\frac{\rho}{4L_0} \right) ^2$,
 	\begin{align*}
 		I_1\Delta 
 		\leq &\bigg[ 2\frac{-L_1(1+\left| X^{i,N}(t_{k})\right|^q)\left| X^{i,N}(t_{k})\right|^2+\left\langle X^{i,N}(t_{k}),b_1(0) \right\rangle }{1+\sqrt{\Delta}\left| X^{i,N}(t_{k})\right|^q}+\left| X^{i,N}(t_{k})\right|^2+2\left|b_2(0,\delta_0) \right|^2\\
 		&\quad+2L_5\left(\left| X^{i,N}(t_{k})\right|^2+\mathbb{W}_2^2(\mu_{t_k}^{X,N},\delta_0) \right) +\frac{4L_0(1+\left| X^{i,N}(t_{k})\right|^{2q})\left| X^{i,N}(t_{k})\right|^2 }{\left( 1+\sqrt{\Delta}\left| X^{i,N}(t_{k})\right|^q\right) ^2}\Delta\\
 		&\quad+4\left|b_1(0) \right| ^2\Delta+4L_5\left(\left| X^{i,N}(t_{k})\right|^2+\mathbb{W}_2^2(\mu_{t_k}^{X,N},\delta_0) \right)\Delta+4\left|b_2(0,\delta_0) \right|^2\Delta\bigg] \Delta\\
 		\leq&\bigg[ -2\left(L_1-2L_0\sqrt{\Delta} \right) \frac{1+\left| X^{i,N}(t_{k})\right|^q}{1+\sqrt{\Delta}\left| X^{i,N}(t_{k})\right|^q}\left| X^{i,N}(t_{k})\right|^2+2(1+3L_5)\left| X^{i,N}(t_{k})\right|^2\\
 		&\quad+6L_5\mathbb{W}_2^2(\mu_{t_k}^{X,N},\delta_0)+C\bigg] \Delta\\
 		\leq&\left[  -2\left(L_1-2L_0\sqrt{\Delta}-1-3L_5 \right)\left| X^{i,N}(t_{k})\right|^2+6L_5\mathbb{W}_2^2(\mu_{t_k}^{X,N},\delta_0)+C\right]  \Delta,
 	\end{align*}
 	where $C=5\left|b_1(0) \right|^2+6\left|b_2(0,\delta_0) \right|^2$ and $L_1-2L_0\sqrt{\Delta}\geq 0 $.
 	Substituting this into \eqref{*}, we deduce that
 	\begin{align*}
 		\left| X^{i,N}(t_{k+1}) \right|^2
 	\leq&\left[1-2\left(L_1-2L_0\sqrt{\Delta}-1-3L_5 \right)\Delta \right] \left| X^{i,N} (t_{k}) \right|^2\\
 	&+6L_5\Delta\mathbb{W}_2^2(\mu_{t_k}^{X,N},\delta_0)+ C\Delta+I_2.
 	\end{align*}
 	\par  	Taking conditional expectation on both sides and applying \eqref{5.2} yield that
 	\begin{align*}
 		\mathbb{E}\left(\big|X^{i,N} (t_{k+1})\big|^{2} \big| \mathcal{F}_{t_{k}}\right)
 				\leq& \left[1-2\left(L_1-2L_0\sqrt{\Delta}-1-3L_5-\left( m+1\right) \lambda_{3} \right)\Delta \right] \left| X^{i,N} (t_{k}) \right|^2\\
 				&+2\left( 3L_5+\left( m+1\right) \lambda_{3}\right) \Delta \mathbb{W}_2^2(\mu_{t_k}^{X,N},\delta_0)+C\Delta,
 	\end{align*} 
 	where we use 
 	\begin{equation*}
 		\begin{split}
 			&\mathbb{E}\left( \triangle B^{i}_k\middle| \mathcal{F}_{t_{k}}\right)  = 0,\quad
 			\mathbb{E}\left( |\triangle B^{i}_k|^{2} \middle|\mathcal{F}_{t_{k}}\right)  =m \Delta,\\
 			&\mathbb{E}\left(\int_{t_{k}}^{t_{k+1}}\int_{Z} \gamma\left( X^{i,N}(t_{k}),\mu_{t_k}^{X,N},z\right) \widetilde{N}^i(\mathrm{d}s,\mathrm{d}z) \middle| \mathcal{F}_{t_{k}}\right)  =0,\\
 			&\mathbb{E}\left(\left| \int_{t_{k}}^{t_{k+1}}\int_{Z} \gamma\left( X^{i,N}(t_{k}),\mu_{t_k}^{X,N},z\right) \widetilde{N}^i(\mathrm{d}s,\mathrm{d}z)\right| ^2 \middle| \mathcal{F}_{t_{k}}\right) \\
 			=& \int_{Z}\left| \gamma\left( X^{i,N}(t_{k}),\mu_{t_k}^{X,N},z\right)\right| ^2  \nu(\mathrm{d}z)\Delta.
 		\end{split}
 	\end{equation*}
 Since $\mathbb{E}\mathbb{W}_2^2(\mu_{t_k}^{X,N},\delta_0)\leq \mathbb{E}\left( \frac{1}{N}\sum_{j=1}^{N}\left| X^{j,N}(t_k)\right|^2 \right)$ and $ X^{i,N}(t_k), i=1,2,\cdots, N  $ are identically distributed, taking expectations again leads to
 \begin{align*}
 	\mathbb{E}\big|X^{i,N} (t_{k+1})\big|^{2}\leq
 \left[1-2\left(\rho-2L_0\sqrt{\Delta}\right)\Delta \right]\mathbb{E}\left| X^{i,N} (t_{k}) \right|^2+C\Delta,
 \end{align*} 
 where $\rho=L_1-1-6L_5-2\left( m+1\right) \lambda_{3} >0$. Due to $\Delta\leq 1 \land \left(\frac{\rho}{4L_0} \right) ^2\land \frac{1}{\rho}$,  $ \rho-2L_0\sqrt{\Delta}\geq \frac{\rho}{2}>0 $,  and for any $k\in \mathbb{N}_0$,
 	\begin{align*}
 		\mathbb{E}\big|X^{i,N} (t_{k+1})\big|^{2}\leq&\left( 1-\rho\Delta \right) \mathbb{E} \left| X^{i,N} (t_{k}) \right|^2+C\Delta \\
 		\leq&\left( 1-\rho\Delta \right)^{k+1} \mathbb{E} \left| X^{i,N} (0) \right|^2+C\Delta\sum_{j=0}^{k} \left( 1-\rho\Delta \right)^j\\
 		\leq & \mathbb{E}\left|  \xi^i\right| ^2+ \frac{C}{\rho}.
 	\end{align*}
 	This completes the
 	proof.
 \end{proof}

 Next, before establishing the exponential contractivity in the $ L^2 $ sense for different initial distributions, we need the following lemma. 
  \begin{lemma}\label{reis}
 	Let Assumption  \ref{dos} hold. Then
 	\begin{align*}
 &	2\left\langle x_1-\!x_2,b_\Delta(x_1,\mu_1)-\!b_\Delta(x_2,\mu_2) \right\rangle+\left|b_\Delta(x_1,\mu_1)\!-b_\Delta(x_2,\mu_2) \right| ^2\Delta\\
 		\leq& -\left( (L_1 \land 2L_3)-4(3L_0+4L_4)\sqrt{\Delta}-1-3L_5\right) |x_{1}\!-x_{2}|^{2}+3L_5\mathbb{W}_2^2(\mu_1,\mu_2)
 	\end{align*}
 for any $ x_1,x_2\in \mathbb{R}^d $, $
 \mu_1,\mu_2 \in \mathcal P_2(\mathbb{R}^d)  $
 	and $\Delta\leq 1\land \left(\frac{L_1}{2L_2} \right)^2$.
 \end{lemma}
 \begin{proof}
 	Due to \eqref{newtamefuction} and Assumption \ref{dos}, one observes that
 	\begin{align*} 
 		&	\left\langle x_1-x_2,b_\Delta(x_1,\mu_1)-b_\Delta(x_2,\mu_2) \right\rangle\\
 		\leq&\frac{-L_1\left(1+\!\left|x_1 \right| ^q+\!\left|x_2 \right| ^q \right)\left| x_1-\!x_2\right|^2  }{\left(  1+\!\sqrt{\Delta} \left| x_1\right|^q\right)\left(  1+\!\sqrt{\Delta}\left| x_2\right|^q \right)  } +\!\frac{\sqrt{\Delta}\left[ L_2\left(1+\!\left|x_1 \right| ^q+\!\left|x_2 \right| ^q \right)-\!L_3\left|x_1 \right| ^q\left|x_2 \right| ^q\right] \left| x_1\!-x_2\right|^2  }{\left(  1+\!\sqrt{\Delta} \left| x_1\right|^q \right) \left(  1+\!\sqrt{\Delta} \left| x_2\right|^q\right)  }\\
 		&+\frac{1}{2}\left| x_1-x_2\right|^2+\frac{L_5}{2}\left(\left| x_1-\!x_2\right|^2+\mathbb{W}_2^2(\mu_1,\mu_2) \right) \\
 		\leq &-\frac{\left( L_1-\sqrt{\Delta}L_2\right) \left(1+\left|x_1 \right| ^q+\left|x_2 \right| ^q \right)+L_3\sqrt{\Delta}\left|x_1 \right| ^q\left|x_2 \right| ^q }{ \left(  1+\!\sqrt{\Delta} \left| x_1\right|^q\right)\left(  1+\!\sqrt{\Delta}\left| x_2\right|^q \right)  }\left| x_1-x_2\right|^2+\frac{1+L_5}{2}\left| x_1-x_2\right|^2\\
 		&+\frac{L_5}{2}\mathbb{W}_2^2(\mu_1,\mu_2)\\
 		\leq &-\frac{\left( \frac{L_1}{2} \land L_3\right)  \left(1+\left|x_1 \right| ^q+\left|x_2 \right| ^q +\sqrt{\Delta}\left|x_1 \right| ^q\left|x_2 \right| ^q\right) }{ \left(  1+\!\sqrt{\Delta} \left| x_1\right|^q\right)\left(  1+\!\sqrt{\Delta}\left| x_2\right|^q \right)  }\left| x_1-x_2\right|^2+\frac{1+L_5}{2}\left| x_1-x_2\right|^2\\
 		&+\frac{L_5}{2}\mathbb{W}_2^2(\mu_1,\mu_2),
 	\end{align*}
 	where we use $\sqrt{\Delta}\leq \frac{L_1}{2L_2}  $. Furthermore,
 	\begin{align*} 
 		&\left|b_\Delta(x_1,\mu_1)-b_\Delta(x_2,\mu_2)\right| ^2\Delta\\
 		\leq&\Bigg[ \frac{4\left|b_1(x_1)\!-b_1(x_2) \right| ^2+\!4\Delta\left|b_1(x_1)\left|x_2 \right| ^q \!-b_1(x_2)\left|x_1 \right| ^q\right| ^2 }{\left| \left(  1+\!\sqrt{\Delta}\left| x_1\right|^q\right) \left(  1+\!\sqrt{\Delta}\left| x_2\right|^q \right) \right| ^2 }\!+2L_5\left(\left| x_1-x_2\right|^2+\mathbb{W}_2^2(\mu_1,\mu_2) \right)\Bigg] \Delta\\
 	\leq&\left[ \frac{4L_0 \left( 1+\left| x_1\right|^{2q}+\left| x_2\right|^{2q}\right)\left| x_1-x_2\right|^2 }{  \left| \left(  1+\!\sqrt{\Delta}\left| x_1\right|^q\right) \left(  1+\!\sqrt{\Delta}\left| x_2\right|^q \right)\right| ^2}+2L_5\left(\left| x_1-x_2\right|^2+\mathbb{W}_2^2(\mu_1,\mu_2) \right)\right.\\
 		&\quad\left.+\frac{4L_4\Delta\left(1+\left|x_1 \right| ^{2q}+\left|x_2 \right| ^{2q}+\left| x_1\right|^{2q}\left| x_2\right|^{2q} \right)\left| x_1-x_2\right|^2  }{ \left| \left(  1+\!\sqrt{\Delta}\left| x_1\right|^q\right) \left(  1+\!\sqrt{\Delta}\left| x_2\right|^q \right)\right| ^2}\right] \Delta\\
 		\leq&\frac{12L_0 \sqrt{\Delta}\left( 1+\left| x_1\right|^{q}+\left| x_2\right|^{q}+\sqrt{\Delta}\left| x_1\right|^{q}\left| x_2\right|^{q}\right) }{ \left(  1+\!\sqrt{\Delta}\left| x_1\right|^q\right) \left(  1+\!\sqrt{\Delta}\left| x_2\right|^q \right)}\left| x_1-x_2\right|^2+2L_5\Delta\left(\left| x_1-x_2\right|^2+\mathbb{W}_2^2(\mu_1,\mu_2) \right)\\
 		&+\frac{16L_4\sqrt{\Delta}\left(1+\left|x_1 \right| ^{q}+\left|x_2 \right| ^{q}+\sqrt{\Delta}\left| x_1\right|^{q}\left| x_2\right|^{q} \right)}{  \left(  1+\!\sqrt{\Delta}\left| x_1\right|^q\right) \left(  1+\!\sqrt{\Delta}\left| x_2\right|^q \right)}\left| x_1-x_2\right|^2  \\
 		\leq&4(3L_0+4L_4)\sqrt{\Delta}\frac{1+\left| x_1\right|^{q}+\left| x_2\right|^{q}+\sqrt{\Delta}\left| x_1\right|^{q}\left| x_2\right|^{q} }{ \left(  1+\!\sqrt{\Delta}\left| x_1\right|^q\right) \left(  1+\!\sqrt{\Delta}\left| x_2\right|^q \right)}\left| x_1-x_2\right|^2\\
 		&+2L_5\left(\left| x_1-x_2\right|^2+\mathbb{W}_2^2(\mu_1,\mu_2) \right).
 	\end{align*}
 	In the above estimate, one uses $\Delta^2\leq \Delta $ and the inequality $\Delta a^{2q}\leq \sqrt{\Delta} a^q(1+\sqrt{\Delta} a^q)$ for all constant $a>0$. Finally, one can
 	observe
 	\begin{align*}
 		&2\left\langle x_1-\!x_2,b_\Delta(x_1,\mu_1)-\!b_\Delta(x_2,\mu_2) \right\rangle+\left|b_\Delta(x_1,\mu_1)\!-b_\Delta(x_2,\mu_2) \right| ^2\Delta\\
 		\leq&-\left((L_1\land 2L_3)-4(3L_0+4L_4)\sqrt{\Delta} \right) \frac{1+\left| x_1\right|^{q}+\left| x_2\right|^{q}+\sqrt{\Delta}\left| x_1\right|^{q}\left| x_2\right|^{q} }{ \left(  1+\!\sqrt{\Delta}\left| x_1\right|^q\right) \left(  1+\!\sqrt{\Delta}\left| x_2\right|^q \right)}\left| x_1-x_2\right|^2\\
 		&+(1+3L_5)\left| x_1-x_2\right|^2+3L_5\mathbb{W}_2^2(\mu_1,\mu_2)\\
 		\leq&-\left((L_1\land 2L_3)-4(3L_0+4L_4)\sqrt{\Delta} -1-3L_5\right)\left| x_1-x_2\right|^2+3L_5\mathbb{W}_2^2(\mu_1,\mu_2),
 	\end{align*}
 	which completes the proof.
 \end{proof}
 \begin{lemma}\label{4.3}  
 	Let Assumptions \ref{g}  and  \ref{dos} hold with $(L_1\land 2L_3) -1-6L_5-2(m+1)\lambda_3>0.$ 
 	There exists a constant $\bar{\lambda}^*>0$ such that for any initial distributions \(\mu,\ \sigma\in\mathcal{P}_{2}(\mathbb{R}^d)\) and $\Delta\leq 1\land \left(\frac{L_1}{2L_2} \right)^2\land \left( \frac{\bar{\lambda}^*}{8(3L_0+4L_4)}\right) ^2 \land \frac{1}{\bar{\lambda}^*}$,
 	\[
 	\mathbb{W}^{2}_{2}\left(P^{\Delta,N}_{t_{k}}\mu,P^{\Delta ,N}_{t_{k}}\sigma\right)\leq  \mathbb{W}^{2}_{2}\left(\mu,\sigma\right)e^{ -\bar{\lambda}^*t_{k}},\quad k\in \mathbb{N}_0.
 	\]
 \end{lemma}
 
 \begin{proof} Let $X^{i,N}(t_k) $ and $Y^{i,N}(t_k)$ be the numerical solutions to scheme \eqref{discrete} such that \(\mathcal{L}(X^{i,N}(0))=\mu \), \(\mathcal{L}(Y^{i,N}(0))=\sigma \) and
 	\begin{align}\label{w}
 		\mathbb{W}^{2}_{2}\left(\mu,\sigma\right)=\mathbb{E}| X^{i,N}(0)-Y^{i,N}(0)|^2.
 	\end{align}
 	
 	\par For brevity, we define for any $k\in \mathbb{N}_0$,
 	\begin{align*}
 		F_\Delta(X^{i,N} (t_k), Y^{i,N} (t_k))\! =& b_\Delta(X^{i,N} (t_k), \mu_{t_k}^{X,N}) \!- b_\Delta(Y^{i,N} (t_k),\mu_{t_k}^{Y,N}),\\
 		G(X^{i,N} (t_k), Y^{i,N} (t_k)) \!= &g(X^{i,N} (t_k),\mu_{t_k}^{X,N}) \!- g(Y^{i,N} (t_k), \mu_{t_k}^{Y,N}),\\
 		\Gamma(X^{i,N} (t_k), Y^{i,N} (t_k)) \!=& \gamma(X^{i,N} (t_k),\mu_{t_k}^{X,N}, z) \!- \gamma(Y^{i,N} (t_k),\mu_{t_k}^{Y,N},z).
 	\end{align*}
 	It follows from scheme \eqref{discrete} that
 	\begin{align*}
 		X^{i,N} (t_{k+1}) - Y^{i,N} (t_{k+1}) =& X^{i,N} (t_{k}) - Y^{i,N} (t_{k}) + F_\Delta(X^{i,N} (t_{k}), Y^{i,N} (t_{k}))  \Delta \\
 		&+G(X^{i,N} (t_{k}), Y^{i,N} (t_{k}))  \Delta B^{i}_k\\
 		&+\int_{t_k}^{t_{k+1}}\!\int_{Z} \Gamma(X^{i,N} (t_{k}), Y^{i,N} (t_{k}))\widetilde{N}^i(\mathrm{d}t,\mathrm{d}z).
 	\end{align*}
 	Simultaneously squaring both sides of the equation yields
 	\begin{align*}
 		&\big|X^{i,N} (t_{k+1}) - Y^{i,N} (t_{k+1})\big|^{2} \\
 		= &\big| X^{i,N} (t_{k}) - Y^{i,N} (t_{k})\big|^{2} + 2\left\langle  X^{i,N} (t_{k}) - Y^{i,N} (t_{k}), F_\Delta(X^{i,N} (t_{k}), Y^{i,N} (t_{k}))  \Delta\right\rangle  \\
 		&+ \big|G(X^{i,N} (t_{k}), Y^{i,N} (t_{k}))\Delta B^{i}_k\big|^{2} + \big|F_\Delta(X^{i,N} (t_{k}), Y^{i,N} (t_{k}))\big|^{2} \Delta^{2} \\
 		&+ 2\left\langle  X^{i,N} (t_{k}) - Y^{i,N} (t_{k}), G(X^{i,N} (t_{k}), Y^{i,N} (t_{k})) \Delta B^{i}_k\right\rangle  \\
 		& + 2\left\langle F_\Delta(X^{i,N} (t_{k}), Y^{i,N} (t_{k}))\Delta, G(X^{i,N} (t_{k}), Y^{i,N} (t_{k})) \Delta B^{i}_k \right\rangle \\
 		&+2\left\langle X^{i,N} (t_{k}) - Y^{i,N} (t_{k}), \int_{t_k}^{t_{k+1}}\int_{Z} \Gamma(X^{i,N} (t_{k}), Y^{i,N} (t_{k}))\widetilde{N}^i(\mathrm{d}t,\mathrm{d}z)\right\rangle \\
 		&+2\left\langle F_\Delta(X^{i,N} (t_{k}), Y^{i,N} (t_{k}))  \Delta, \int_{t_k}^{t_{k+1}}\int_{Z} \Gamma(X^{i,N} (t_{k}), Y^{i,N} (t_{k}))\widetilde{N}^i(\mathrm{d}t,\mathrm{d}z)\right\rangle\\
 		&+\!2\left\langle G(X^{i,N} (t_{k}), Y^{i,N} (t_{k}))\Delta B^{i}_k, \int_{t_k}^{t_{k+1}}\int_{Z}\! \Gamma(X^{i,N} (t_{k}), Y^{i,N} (t_{k}))\widetilde{N}^i(\mathrm{d}t,\mathrm{d}z)\right\rangle\\
 		&+\bigg| \int_{t_k}^{t_{k+1}}\int_{Z} \Gamma(X^{i,N} (t_{k}), Y^{i,N} (t_{k}))\widetilde{N}^i(\mathrm{d}t,\mathrm{d}z)\bigg|^{2}. 
 	\end{align*}
 	Taking conditional expectation on both sides, using Lemma \ref{reis} and Assumption \ref{g}, for any $\Delta\leq 1\land \left(\frac{L_1}{2L_2} \right)^2$, we derive
 	\begin{align*}
 		&\mathbb{E}\left(\big|X^{i,N} (t_{k+1}) - Y^{i,N} (t_{k+1})\big|^{2} \big| \mathcal{F}_{t_{k}}\right) \\
 		\leq&\big| X^{i,N} (t_{k}) - Y^{i,N} (t_{k})\big|^{2} + 2\left\langle  X^{i,N} (t_{k}) - Y^{i,N} (t_{k}), F_\Delta(X^{i,N} (t_{k}), Y^{i,N} (t_{k}))  \right\rangle \Delta \\
 		&+ \big|F_\Delta(X^{i,N} (t_{k}), Y^{i,N} (t_{k}))\big|^{2} \Delta^{2}+ \big|G(X^{i,N} (t_{k}), Y^{i,N} (t_{k}))\big|^{2}m\Delta\\
 		&+\int_{Z} \left| \Gamma(X^{i,N} (t_{k}), Y^{i,N} (t_{k}))\right|^2 \nu(\mathrm{d}z)\Delta\\
 		\leq& \big|X^{i,N} (t_{k}) - Y^{i,N} (t_{k})\big|^{2}+3L_5\mathbb{W}_2^2(\mu_{t_k}^{X,N},\mu_{t_k}^{Y,N}) \Delta\\
 		&-\left((L_1\land 2L_3)-4(3L_0+4L_4)\sqrt{\Delta} -1-3L_5\right)\big|X^{i,N} (t_{k}) - Y^{i,N} (t_{k})\big|^{2}\Delta\\
 		&+(m+1)\lambda_3 \left(\big|X^{i,N} (t_{k}) -\! Y^{i,N} (t_{k})\big|^{2}+ \mathbb{W}_2^2(\mu_{t_k}^{X,N},\mu_{t_k}^{Y,N})\right) \Delta.
 	\end{align*} 
 	Taking expectation on both sides again, since $ X^{i,N}(t_k),i=1,2,\cdots, N  $ are identically distributed, we have
 	\begin{align*}
 		\mathbb{E}\big|X^{i,N} (t_{k+1}) - Y^{i,N} (t_{k+1})\big|^{2} 
 		\leq \left[ 1-\left(2\bar{\lambda}^*-4(3L_0+4L_4)\sqrt{\Delta} \right)\Delta\right]\mathbb{E}\big|X^{i,N} (t_{k}) - Y^{i,N} (t_{k})\big|^{2},
 	\end{align*} 
 	 where $\bar{\lambda}^*:=\frac{1}{2}\left[ (L_1\land 2L_3) -1-6L_5-2(m+1)\lambda_3\right] >0 $. Due to $\Delta\leq 1\land \left( \frac{\bar{\lambda}^*}{8(3L_0+4L_4)}\right) ^2 \land \frac{1}{\bar{\lambda}^*}$, $2\bar{\lambda}^*-4(3L_0+4L_4)\sqrt{\Delta} \geq \frac{3\bar{\lambda}^*}{2}\geq\bar{\lambda}^*$, using the inequality $1-x\leq e^{-x}$, one obtains 
 	\begin{align}\label{eq1}
 		\mathbb{E}\big|X^{i,N} (t_{k+1}) - Y^{i,N} (t_{k+1})\big|^{2}
 		\leq& \left(  1-\bar{\lambda}^*\Delta\right) \mathbb{E}\big|X^{i,N} (t_{k}) - Y^{i,N} (t_{k})\big|^{2}\notag\\
 		\leq&\left(  1-\bar{\lambda}^*\Delta\right) ^{k+1}\mathbb{E}\big|X^{i,N} (0) - Y^{i,N} (0)\big|^{2}\notag\\
 		\leq&e^{-\bar{\lambda}^*t_{k+1}}\mathbb{E}\big|X^{i,N} (0) - Y^{i,N} (0)\big|^{2}.
 	\end{align} 
 	The proof is complete with \eqref{w}.
 \end{proof}
 Now we give the exponential ergodicity of tamed Euler method \eqref{discrete}.
 \begin{theorem}\label{4.5}
 	Let Assumptions \ref{2.1},\ref{g} and \ref{dos} hold with  $(L_1\land 2L_3) -1-6L_5-2(m+1)\lambda_3>0.$
 	For any $\Delta\in (0,\Delta^*] $, $\Delta^*:= 1\land \left(\frac{L_1}{2L_2} \right)^2\land \left( \frac{\bar{\lambda}^*}{8(3L_0+4L_4)}\right) ^2 \land \left(\frac{\rho}{4L_0} \right) ^2\land \frac{1}{\rho}$, there exists a unique invariant probability measure 	$\Pi^{\Delta,N}\in\mathcal{P}_2((\mathbb{R}^d)^N)$	to the tamed Euler scheme \eqref{discrete} satisfying that
 	\[
 	\mathbb{W}_{2}^2\bigl(\mathbf{P}^{\Delta,N}_{t_k}\mu^{\otimes N},\,\Pi^{\Delta,N}\bigr)
 	\leq e^{-\bar{\lambda}^* t_k}\mathbb{W}_{2}^2\bigl(\mu^{\otimes N},\Pi^{\Delta,N}\bigr),
 	\]
 	where $\rho$ and  $\bar{\lambda}^* $ are given in Lemma \ref{new1} and \ref{4.3}, respectively.
 \end{theorem}
 
 \begin{proof}
 	Let $\mu,\sigma\in\mathcal P_2(\mathbb R^d)$, then $\mu^{\otimes N},\sigma^{\otimes N}\in\mathcal P_2\left( (\mathbb R^d)^N\right) $, similar as Lemma \ref{4.3}, set $X^{i,N}(t_k) $ and $Y^{i,N}(t_k)$ be the numerical solutions to scheme \eqref{discrete} such that \(\mathcal{L}(X^{i,N}(0))=\mu \), \(\mathcal{L}(Y^{i,N}(0))=\sigma \) and satisfy \eqref{w}. Since initial particles are i.i.d, we have 
 	\begin{align*}
 		(X^{1,N}(0),\dots,X^{N,N}(0))\sim \mu^{\otimes N}, \qquad (Y^{1,N}(0),\dots,Y^{N,N}(0))\sim \sigma^{\otimes N},
 	\end{align*}
 	thus $ \mathbb{W}_{2}^2(\mu^{\otimes N},\sigma^{\otimes N})= \mathbb E\left( \sum_{i=1}^N|X^{i,N} (0) - Y^{i,N} (0)|^2\right) =N \mathbb{W}_2^2(\mu,\sigma).$ Using \eqref{eq1} one has
 	\begin{align}\label{contraction_joint}
 		\mathbb{W}_{2}^2\bigl(\mathbf{P}^{\Delta,N}_{t_k}\mu^{\otimes N},\,
 		\mathbf{P}^{\Delta,N}_{t_k}\sigma^{\otimes N}\bigr)\leq& \mathbb E\left( \sum_{i=1}^N\big|X^{i,N} (t_{k}) - Y^{i,N} (t_{k})\big|^{2}\right) \notag\\
 		\leq&e^{-\bar{\lambda}^*t_{k}}N\mathbb{E}\big|X^{i,N} (0) - Y^{i,N} (0)\big|^{2}\notag\\
 		\leq& e^{-\bar{\lambda}^* t_k}\mathbb{W}_{2}^2\bigl(\mu^{\otimes N},\sigma^{\otimes N}\bigr),
 		\qquad k\in\mathbb{N}_0.
 	\end{align}
 	
 	For any $k,n\in\mathbb{N}_0$, using \eqref{contraction_joint}, we obtain
 	\begin{align}
 		\mathbb{W}_{2}^2\bigl(\mathbf{P}^{\Delta,N}_{t_{k+n}}\mu^{\otimes N},\,
 		\mathbf{P}^{\Delta,N}_{t_k}\mu^{\otimes N}\bigr) 
 		&= \mathbb{W}_{2}^2\bigl(\mathbf{P}^{\Delta,N}_{t_k}(\mathbf{P}^{\Delta,N}_{t_n}\mu^{\otimes N}),\,
 		\mathbf{P}^{\Delta,N}_{t_k}\mu^{\otimes N}\bigr) \notag \\
 		&\leq e^{-\bar{\lambda}^* t_k}\mathbb{W}_{2}^2\bigl(\mathbf{P}^{\Delta,N}_{t_n}\mu^{\otimes N},\,
 		\mu^{\otimes N}\bigr)\notag \\
 		&\leq e^{-\bar{\lambda}^* t_k}\mathbb{E}\left( \sum_{i=1}^{N}\left|  X^{i,N}(t_{n})-X^{i,N}(0)\right|  ^{2}\right)\notag \\
 		&\leq 2Ne^{-\bar{\lambda}^* t_k}\left( \mathbb{E}\left|  X^{i,N}(t_{n})\right|  ^{2}+ \mathbb{E}\left|  X^{i,N}(0)\right|  ^{2}\right).\notag
 	\end{align}
 	By Lemma \ref{new1}, we derive that
 	\[
 	\lim_{k\to\infty}\sup_{n\in\mathbb{N}_0}\mathbb{W}_{2}^2\bigl(\mathbf{P}^{\Delta,N}_{t_{k+n}}\mu^{\otimes N},\,
 	\mathbf{P}^{\Delta,N}_{t_k}\mu^{\otimes N}\bigr)= 0.
 	\]
 	Thus $\{\mathbf{P}^{\Delta,N}_{t_k}\mu^{\otimes N}\}_{k\in\mathbb{N}_0}$ is a Cauchy sequence 
 	in the complete metric space $\bigl(\mathcal{P}_2((\mathbb{R}^d)^N),\,\mathbb{W}_{2}\bigr)$.
 	Consequently, there exists a unique measure $\Pi^{\Delta,N}\in\mathcal{P}_2((\mathbb{R}^d)^N)$ 
 	such that
 	\[
 	\lim_{k\to\infty}\mathbb{W}_{2}\bigl(\mathbf{P}^{\Delta,N}_{t_k}\mu^{\otimes N},\,
 	\Pi^{\Delta,N}\bigr)=0.
 	\]
 	Then for any \( n\in \mathbb{N}_0\), by triangle inequality,
 	\begin{align*}
 		\mathbb{W}_{2}\bigl(\mathbf{P}^{\Delta,N}_{t_n}\Pi^{\Delta,N},\,
 		\Pi^{\Delta,N}\bigr)\leq&\lim_{k\to\infty}\mathbb{W}_{2}\bigl(\mathbf{P}^{\Delta,N}_{t_{n+k}}\mu^{\otimes N},\mathbf{P}^{\Delta,N}_{t_{k}}\mu^{\otimes N}\bigr)=0.
 	\end{align*}
 	Thus,  $\Pi^{\Delta,N}$ is the invariant measure of the tamed Euler scheme \eqref{discrete}.
 	For any initial measure $\mu\in\mathcal{P}_2(\mathbb{R}^d)$, $\Delta\in (0,\Delta^*] $  and \( k\in \mathbb{N}_0\), using \eqref{contraction_joint},
 	\begin{align*}
 		\mathbb{W}_{2}^2\bigl(\mathbf{P}^{\Delta,N}_{t_k}\mu^{\otimes N},\,
 		\Pi^{\Delta,N}\bigr)=&\mathbb{W}_{2}^2\bigl(\mathbf{P}^{\Delta,N}_{t_k}\mu^{\otimes N},\,
 		\mathbf{P}^{\Delta,N}_{t_k}\Pi^{\Delta,N}\bigr)\notag\\
 		\leq& e^{-\bar{\lambda}^* t_k}\mathbb{W}_{2}^2\bigl(\mu^{\otimes N},\Pi^{\Delta,N}\bigr).
 	\end{align*}
 	The proof is complete.
 \end{proof}
 \section{Convergence of numerical invariant measure}
 This section proves that the numerical invariant measure of the tamed Euler scheme \eqref{discrete} converges to the invariant measure of the underlying MV-SDE \eqref{eq1.1}. Since the numerical invariant measure is defined on $(\mathbb R^d)^N$ while the invariant measure of the MV-SDE is on $\mathbb R^d$, we proceed in two steps. First, we compare the numerical invariant measure with its $N$-tensorised  exact counterpart. Then, by considering the marginal distribution, which is also the distribution of a single particle, we obtain the desired convergence. To achieve this, we need to quantify  the convergence of PoC and the time discretization over a finite time horizon $T>0$.  In addition, we suppose:
 \begin{assumption}\label{p}
 	There exists a constant \(p_0>2\) such that $	\mathbb{E}\left| \xi \right|^{p_0} <\infty$.
 \end{assumption}
 \begin{assumption}\label{GAMMA}
 	There exists a constant $ L> 0$ such that for any $ x\in \mathbb{R}^d $ and $
 	\mu \in \mathcal P_2(\mathbb{R}^d)  $,
 	\begin{eqnarray*}
 		\left| g(x,\mu)\right|^{p_0}\vee	\int_{Z}\!\left| \gamma(x,\mu,z)\right|^{p_0}\! \nu(\mathrm{d}z)\leq\! L \left( 1+\left|  x\right|^{p_0}+\mathbb{W}_2^{p_0}(\mu,\delta_0)\right),
 	\end{eqnarray*}
 	where $p_0$ is given in Assumption \ref{p}.
 \end{assumption}
 
 \subsection{Propagation of chaos}
 To begin with, we provide the moment bound of solutions for IPS \eqref{eq2.7}.
 \begin{lemma}\label{2.5}
 	Under Assumptions \ref{dos}, \ref{p} and \ref{GAMMA}, the following bounds hold:
 	\begin{eqnarray*}
 		\sup_{1 \leq i \leq N}\mathbb{E}\left( \sup_{0 \leq t \leq T}\left|x^{i,N}(t) \right|^{p} \right)  \leq C \quad\text{for any $p\in[2,p_0]$},
 	\end{eqnarray*}
 	where $C:=C(T,p,p_0,L_1,L_5,\lambda_3,
 	\mathbb{E}|\xi^i|^p)>0  $ and $p_0$ is given in Assumption \ref{p}.
 \end{lemma}
 \begin{proof}
 	Applying  It{\^o}'s formula to $ \left|x^{i,N}(t)\right| ^{p} $ and taking the expectation of its supremum over any $ T_1\in[0,T] $, we obtain
 	\begin{align*}
 		\mathbb{E}\left( \sup_{0 \leq t \leq T_1} |x^{i,N}(t)|^p \right) 
 		\leq&\mathbb{E}\left| \xi^i\right|^p + p\mathbb{E}\left( \sup_{0 \leq t \leq T_1} \int_{0}^{t} |x^{i,N}(s)|^{p-2} \langle x^{i,N}(s), b(x^{i,N}(s),\mu_s^{x,N} )\rangle \mathrm{d}s \right) \notag \\
 		& + p\mathbb{E}\left( \sup_{0 \leq t \leq T_1} \int_{0}^{t} |x^{i,N}(s)|^{p-2} \langle x^{i,N}(s), g(x^{i,N}(s),\mu_s^{x,N} )\rangle \mathrm{d}B^i(s) \right) \notag \\
 		&+ \frac{p(p-1)}{2}\mathbb{E}\left( \sup_{0 \leq t \leq T_1} \int_{0}^{t} |x^{i,N}(s)|^{p-2} |g(x^{i,N}(s),\mu_s^{x,N})|^2 \mathrm{d}s \right) \notag \\
 		& + p\mathbb{E}\left( \sup_{0 \leq t \leq T_1} \int_{0}^{t} \int_{Z} \!|x^{i,N}(s)|^{p-2} \langle x^{i,N}(s), \gamma(x^{i,N}(s),\mu_s^{x,N},z) \rangle \widetilde{N}^i(\mathrm{d}s,\mathrm{d}z)\! \right) \notag \\
 		& + \mathbb{E}\bigg( \sup_{0 \leq t \leq T_1} \int_{0}^{t} \int_{Z} \Big( |x^{i,N}(s) + \gamma(x^{i,N}(s),\mu_s^{x,N},z)|^p - |x^{i,N}(s)|^p \notag \\
 		&\quad \quad - p|x^{i,N}(s)|^{p-2} \langle x^{i,N}(s), \gamma(x^{i,N}(s),\mu_s^{x,N},z) \rangle \Big) N^i(\mathrm{d}s,\mathrm{d}z) \bigg).
 	\end{align*}
 	Using Kunita's inequality (Lemma 2.1 in \cite{Kumar2016}), Young's inequality, H{\"o}lder's inequality	and the following remainder formula:
 	\begin{align}\label{eq2.5}
 		\left|y \right|^{p}-\left|b \right|^{p}-p\left|b \right|^{p-2}\left\langle b,y-b \right\rangle=&p(p-1)\int_{0}^{1}(1-\theta)\left|y-b \right|^2\left|b+\theta(y-b) \right|^{p-2}  \mathrm{d}\theta \notag\\
 		\leq &C\left(\left|b \right|^{p-2}\left| y-b\right| ^2+\left| y-b\right| ^{p}\right)
 	\end{align}
 	for any $ y, b\in \mathbb{R}^d $, we can calculate the last two estimates
 	\begin{align*}
 		&p\mathbb{E}\left( \sup_{0 \leq t \leq T_1} \int_{0}^{t} \int_{Z} |x^{i,N}(s)|^{p\!-2} \langle x^{i,N}(s),  \gamma(x^{i,N}(s),\mu_s^{x,N},z) \rangle \widetilde{N}^i(\mathrm{d}s,\!\mathrm{d}z) \right)  \\
 		&+ \mathbb{E}\bigg( \sup_{0 \leq t \leq T_1} \int_{0}^{t} \int_{Z} \Big( |x^{i,N}(s) +  \gamma(x^{i,N}(s),\mu_s^{x,N},z)|^p - |x^{i,N}(s)|^p\\
 		&\qquad - p|x^{i,N}(s)|^{p-2} \langle x^{i,N}(s),  \gamma(x^{i,N}(s),\mu_s^{x,N},z) \rangle \Big) N^i(\mathrm{d}s,\mathrm{d}z) \bigg) \\
 		\leq &C\mathbb{E}\left(  \int_{0}^{T_1}\int_{Z}\left| x^{i,N}(t)\right| ^{2p-2}\left|  \gamma(x^{i,N}(t),\mu_t^{x,N},z)\right| ^2\nu(\mathrm{d}z)\mathrm{d}t\right) ^{\frac{1}{2}}\\
 		&+C\mathbb{E}\left[   \int_{0}^{T_1}\int_{Z}\left( \left|x^{i,N}(t) \right|^{p-2}\left|  \gamma(x^{i,N}(t),\mu_t^{x,N},z)\right| ^2+\left| \gamma(x^{i,N}(t),\mu_t^{x,N},z)\right| ^{p}\right) \nu(\mathrm{d}z)\mathrm{d}t\right] \\
 		\leq &C\mathbb{E}\left(\sup_{0 \leq t \leq T_1}\left| x^{i,N}(t)\right| ^{p}  \int_{0}^{T_1}\int_{Z}\left| x^{i,N}(t)\right| ^{p-2}\left|  \gamma(x^{i,N}(t),\mu_t^{x,N},z)\right| ^2\nu(\mathrm{d}z)\mathrm{d}t\right) ^{\frac{1}{2}}\\
 		&+C\mathbb{E}  \int_{0}^{T_1}\left[\left| x^{i,N}(t)\right| ^{p} +\left( \int_{Z}\left| \gamma(x^{i,N}(t),\mu_t^{x,N},z)\right| ^{p_0}\nu(\mathrm{d}z)\right) ^{\frac{p}{p_0}}\left( \int_{Z}1~\nu(\mathrm{d}z)\right) ^{\frac{p_0-p}{p_0}}\right] \mathrm{d}t.
 	\end{align*}
 	Note that $\int_{Z}1~\nu(\mathrm{d}z)=\nu(Z)<\infty $. By employing Assumptions \ref{dos}, \ref{GAMMA},	H{\"o}lder's inequality, Young's inequality and the Burkholder-Davis-Gundy (B-D-G) inequality,	we can derive that
 	\begin{align}
 		\!	\mathbb{E}\!\left( \sup_{0 \leq t \leq T_1}\left| x^{i,N}(t)\right| ^{p} \right)
 		\leq& \mathbb{E}\left| \xi^i\right|^{p}+\frac{1}{2}\mathbb{E}\left( \sup_{0 \leq t \leq T_1}\left| x^{i,N}(t)\right| ^{p}\right)\notag\\
 		&+C\mathbb{E}\left[   \int_{0}^{T_1}\left( 1+\left| x^{i,N}(t)\right| ^p+\mathbb{W}_2^p(\mu_t^{x,N},\delta_0)\right)  \mathrm{d}t\right] .\notag
 	\end{align} 
 	Using the Fubini theorem, since $ x^{i,N}(t),i=1,2,\cdots, N  $ are identically distributed, $\mathbb{E}\mathbb{W}_2^p(\mu_t^{x,N},\delta_0)\leq \mathbb{E}\left( \frac{1}{N}\sum_{j=1}^{N}\left| x^{j,N}(t)\right|^p \right)=\mathbb{E}\left|  x^{i,N}(t)\right| ^p $, we obtain for any $ T_1\in [0,T] $,
 	\begin{align}
 		\mathbb{E}\left( \sup_{0 \leq t \leq T_1} |x^{i,N}(t)|^p \right)
 		\leq C\left[ 1+\mathbb{E}\left| \xi^i\right|^p  + \int_{0}^{T_1}\mathbb{E}\left( \sup_{0 \leq s \leq t} |x^{i,N}(s)|^p \right) \mathrm{d}t \right].\notag
 	\end{align} 
 	By Gr{\"o}nwall's inequality, we can get the desired assertion. 
 \end{proof}  
\begin{theorem}\label{2.6}
	Let Assumptions \ref{g}, \ref{dos}, \ref{p} and \ref{GAMMA} hold with $p_0>4$. Then, there exists a constant $C>0$, independent of $N$, such that
	\begin{equation*}\label{eq2.10}
		\sup_{1 \leq i \leq N} \mathbb{E}\!\left( \sup_{0 \leq t \leq T} \left|x^i(t) \!- x^{i,N}(t)\right|^2\right) \leq \!C\varphi(N),
	\end{equation*}
	where
	\begin{equation*}
		\varphi(N):=
		\begin{cases}
			N^{-\frac{1}{2}}, &d<4, \\
			N^{-\frac{1}{2}} \log N, &d=4, \\
			N^{-\frac{2}{d}}, &d>4.
		\end{cases}
	\end{equation*}
\end{theorem}
\begin{proof}
	Define the error $ e^{i,N}(t):=x^i(t)-x^{i,N}(t) $. Applying It{\^o}'s formula to obtain
	\begin{align}\label{eq2.11}
		|e^{i,N}(t)|^2
		\leq& 2\int_{0}^{t} \langle e^{i,N}(s), b(x^i(s),\mu_s^{x^i}) - b(x^{i,N}(s),\mu_s^{x,N}) \rangle \mathrm{d}s \notag \\
		&+ \int_{0}^{t} |g(x^i(s),\mu_s^{x^i}) - g(x^{i,N}(s),\mu_s^{x,N})|^2 \mathrm{d}s \notag \\
		&\! + 2\int_{0}^{t}  \langle e^{i,N}(s), g(x^i(s),\mu_s^{x^i}) - g(x^{i,N}(s),\mu_s^{x,N}) \rangle \mathrm{d}B^i(s) \notag \\
		& + 2\int_{0}^{t} \int_{Z} \langle e^{i,N}(s), \gamma(x^i(s),\mu_s^{x^i},z) - \gamma(x^{i,N}(s),\mu_s^{x,N},z) \rangle \widetilde{N}^i(\mathrm{d}s,\mathrm{d}z) \notag \\
		&+ \int_{0}^{t} \int_{Z} | \gamma(x^i(s),\mu_s^{x^i},z) - \gamma(x^{i,N}(s),\mu_s^{x,N},z)|^2 N^i(\mathrm{d}s,\mathrm{d}z) \notag \\
		=&: J_1 + J_2 + J_3 + J_4 + J_5.
	\end{align}
	\par Define $ \mu_s^N:=\frac{1}{N}\sum_{j=1}^{N}\delta_{x^{j}(s)} $, which is the empirical measure constructed from i.i.d samples of the exact solution, then we have
	\begin{align}\label{w2}
		\mathbb{W}^2_2(\mu_s^{x^i},\mu_s^{x,N})\leq& 2\left( \mathbb{W}^2_2(\mu_s^{x^i},\mu_s^{N})+\mathbb{W}^2_2(\mu_s^N,\mu_s^{x,N})\right) \notag\\
		\leq &2\left( \mathbb{W}^2_2(\mu_s^{x^i},\mu_s^{N})+\frac{1}{N}\sum_{j=1}^{N}\left| x^{j}(s)-x^{j,N}(s)\right| ^2\right).
	\end{align}
	Taking the expectation to the supremum  on both sides of \eqref{eq2.11}, it follows by Remark \ref{remark}, Assumption \ref{GAMMA} and Young's inequality that 	for any $ T_1\in[0,T] $,
	\begin{align}\label{eq2.12}
	\mathbb{E}\left( \sup_{0 \leq t \leq T_1} (	J_1+J_2+J_5)\right) &\leq C\mathbb{E}\int_{0}^{T_1} \bigg( |e^{i,N}(s)|^2 + \mathbb{W}_2^2(\mu_s^{x^i},\mu_s^N) + \frac{1}{N}\sum_{j=1}^N |e^{j,N}(s)|^2 \bigg) \mathrm{d}s.
	\end{align}
	Applying Assumption \ref{g}, the B-D-G inequality, Kunita's inequality,  Young's inequality and H{\"o}lder's inequality, we have
	\begin{align}\label{eq2.16}
		&\mathbb{E}\left(\sup_{0 \leq t \leq T_1} (J_3 + J_4)\right) \notag\\
		\leq& C\left[\mathbb{E}\left(\sup_{0 \leq t \leq T_1}\!|e^{i,N}(t)|^2 \int_{0}^{T_1} |g(x^i(t),\mu_t^{x^i})\! - g(x^{i,N}(t),\mu_t^{x,N})|^2 \mathrm{d}t\right)^{\frac{1}{2}}\right.\notag\\
		&+ \left.\mathbb{E}\left(\sup_{0 \leq t \leq T_1}|e^{i,N}(t)|^2 \int_{0}^{T_1}\int_{Z} |\gamma(x^i(t),\mu_t^{x^i},z)\! - \gamma(x^{i,N}(t),\mu_t^{x,N},z)|^2 \nu(\mathrm{d}z)\mathrm{d}t\right)^{\frac{1}{2}}\right]\notag\\
		\leq &\frac{1}{2}\mathbb{E}\left(\sup_{0 \leq t \leq T_1}|e^{i,N}(t)|^2\right)+C\int_{0}^{T_1}\mathbb{E}\bigg(|e^{i,N}(t)|^2 + \mathbb{W}_2^2(\mu_t^{x^i},\mu_t^N)+ \frac{1}{N}\sum_{j=1}^N |e^{j,N}(t)|^2\bigg)\mathrm{d}t.
	\end{align}

	Thus, combining (\ref{eq2.12}) and (\ref{eq2.16}), since particles $x^{i,N}(t), i=1,\cdots,N$ are identically distributed, we obtain
	\begin{align}
		\mathbb{E}\left( \sup_{0 \leq t \leq T_1}\left|e^{i,N}(t) \right|^{2}\right) 
		\leq C\int_{0}^{T_1}\left[   \mathbb{E}\left( \sup_{0 \leq s \leq t}\left|e^{i,N}(s) \right|^{2}\right) 
		+ \mathbb{E}\mathbb{W}_2^2\!\left(\mu_t^{x^i},\mu_t^{N} \right)
		\right]  \mathrm{d}t.\notag
	\end{align}
	By the use of Gr{\"o}nwall's inequality and the known quantified rate for the PoC theory given in \cite{CarmonaI2018} along with Lemma \ref{2.5}, we can get the final estimate.
\end{proof} 
\begin{remark}\label{sharp}
	\rm	Note that in the present case, the coefficients are assumed to be Lipschitz continuous in $\mathbb{W}_2$-distance with
	respect to the measure variable so that Theorem 5.8 in \cite{CarmonaI2018} is used to estimate the convergence rate of PoC, which depends on the dimension $d$ and seems a little complicated. However,
	if we impose the following assumption on the measure dependence instead of $\mathbb{W}_2$-Lipschitz condition, this result can be improved to be independent of the dimension, i.e.,  there exists a constant $C>0$, independent of $N$, such that
	\begin{equation}\label{eq3.11}
		\mathbb{E}\left( \sup_{0 \leq t \leq T}\left|e^{i,N}(t) \right|^{2}\right) \leq CN^{-\frac{1}{2}}.
	\end{equation} 
	The assumption is: there exists a constant $L>0$ and functions $\phi_u(x,w), u=1,2,\cdots,U,$ continuous on $\mathbb{R}^d\times \mathbb{R}^d$, satisfying for some $L_u>0$,
	\begin{align*}
		\left|\phi_u(x,w_1)-\phi_u(x,w_2) \right| \leq L_u\left|w_1-w_2 \right|, 
	\end{align*}
	such that for any $x\in \mathbb{R}^d$ and $
	\mu_1,\mu_2 \in \mathcal P_2(\mathbb{R}^d)  $,
	\begin{align}
		&\left| (b,g)(x,\mu_1)-(b,g)(x,\mu_2)\right|\vee \int_{Z}\left| \gamma(x,\mu_1,z)-\gamma(x,\mu_2,z)\right|\nu(\mathrm{d}z)\notag\\
		\leq& L\left( \sum_{u=1}^{U}\left| \int_{\mathbb{R}^{d}}\phi_u(x,w)\mu_1(\mathrm{d}w)-\int_{\mathbb{R}^{d}}\phi_u(x,w)\mu_2(\mathrm{d}w)\right|\right).\notag
	\end{align}
	Following the same line as Theorem \ref{2.6}, referring to the techniques in \cite{Hinds2025} and \cite{Yuhang2025}, and employing the Rosenthal inequality (see \cite{Rosenthal1970}), we can  derive \eqref{eq3.11}. In addition, since we no longer need to invoke Theorem 5.8 in \cite{CarmonaI2018}, the initial distribution is only required to have a finite second-order moment in this case (rather than a higher-order moment $p_0>4$ as needed in Theorem \ref{2.6}). The detailed proof is omitted here.
\end{remark}
 
\subsection{Convergence of tamed Euler method}
  In this subsection, we proceed to prove the convergence of the numerical scheme \eqref{eq3.1} to  equation \eqref{eq2.7}. Without loss of generality, assume that there exists a sufficiently large integer $M>1$ such that  $ \Delta:= \frac{T}{M}\in(0,1)$. For any $ k \in \{0,1,\cdots,M\}$,  let $ t_k:=k\Delta $ be discrete time points  and $ \check{t}:=t_k $ for $ t_k\leq t<t_{k+1} $, the continuous version of tamed Euler scheme \eqref{discrete} is defined as
 \begin{align}\label{eq3.1}
 	X^{i,N}(t)=&\xi^i+\int_{0}^{t} b_\Delta(X^{i,N}(\check{s}),\mu_{\check{s}}^{X,N})\mathrm{d}s+\int_{0}^{t} g(X^{i,N}(\check{s}),\mu_{\check{s}}^{X,N})\mathrm{d}B^i(s)\notag\\
 	&+\int_{0}^{t}\int_{Z} \gamma(X^{i,N}(\check{s}),\mu_{\check{s}}^{X,N},z)\widetilde{N}^i(\mathrm{d}s,\mathrm{d}z),
 \end{align}
 where $ \mu_{\check{t}}^{X,N}:=\frac{1}{N}\sum_{j=1}^{N}\delta_{X^{j,N}(\check{t})} $ and $b_\Delta$ is defined as  \eqref{newtamefuction}.

We first introduce two auxiliary lemmas. 
  \begin{lemma}\label{3.1}
 	Under Assumptions \ref{dos},\ref{p} and \ref{GAMMA}, the following estimate holds:
 	\begin{eqnarray*}
 		\sup_{1 \leq i \leq N}\mathbb{E}\left( \sup_{0 \leq t \leq T}\left|X^{i,N}(t) \right|^{p}\right) \leq C \quad \text{for any $p\in[2,p_0]$,}
 	\end{eqnarray*}
 	where $ C:=C(T,p,p_0,L_0,L_1,L_3,L_5,\lambda_3,
 	\mathbb{E}|\xi^i|^p)>0 $ and $p_0$ is given in Assumption \ref{p}.
 \end{lemma}
 \begin{proof}
 	For any $ i=1,2,\cdots, N $, $ T_1\in[0,T] $, same as Lemma \ref{2.5}, from It{\^o}'s formula,
 	\begin{align}\label{eq3.4}
 		&\mathbb{E}\left( \sup_{0 \leq t \leq T_1}\left| X^{i,N}(t)\right| ^{p}\right)\notag \\
 		\leq&\mathbb{E}\left| \xi^i\right|^{p}+\!p\mathbb{E}\left[  \sup_{0 \leq t \leq T_1}\left(\int_{0}^{t}\left| X^{i,N}(s)\right| ^{p-2}\left\langle X^{i,N}(\check{s}),b_{\Delta}(X^{i,N}(\check{s}),\mu_{\check{s}}^{X,N})\right\rangle  \mathrm{d}s\right)\right] \notag\\
 		&+p\mathbb{E}\!\left[  \sup_{0 \leq t \leq T_1}\left( \int_{0}^{t}\!\left| X^{i,N}(s)\right| ^{p-2}\left\langle X^{i,N}(s)-X^{i,N}(\check{s}),b_{\Delta}(X^{i,N}(\check{s}),\mu_{\check{s}}^{X,N})\right\rangle \mathrm{d}s\right)\right] \notag \\
 		&+\!p\mathbb{E}\!\left[  \sup_{0 \leq t \leq T_1}\left( \int_{0}^{t}\left| X^{i,N}(s)\right| ^{p-2}\left\langle X^{i,N}(s),g(X^{i,N}(\check{s}),\mu_{\check{s}}^{X,N})\right\rangle \mathrm{d}B^i(s)\right)\right] \notag\\
 		&+\!\frac{p(p-1)}{2}\mathbb{E}\!\left[ \sup_{0 \leq t \leq T_1}\left( \int_{0}^{t}\left| X^{i,N}(s)\right| ^{p-2}\left|g(X^{i,N}(\check{s}),\mu_{\check{s}}^{X,N})\right|^2\!\mathrm{d}s\right)\right] \notag \\
 		&+\!p\mathbb{E}\left[  \sup_{0 \leq t \leq T_1}\left( \int_{0}^{t}\!\int_{Z}\left| X^{i,N}(s)\right| ^{p-2}\left\langle X^{i,N}(s),\gamma(X^{i,N}(\check{s}),\mu_{\check{s}}^{X,N},z)\right\rangle\widetilde{N}^i(\mathrm{d}s,\mathrm{d}z)\right)\right]  \notag\\
 		&+\!\mathbb{E}\Bigg[ \sup_{0 \leq t \leq T_1}\Bigg( \int_{0}^{t}\int_{Z}\Big( \left| X^{i,N}(s)\!+\gamma(X^{i,N}(\check{s}),\mu_{\check{s}}^{X,N},z)\right| ^{p}-\!\left| X^{i,N}(s)\right| ^{p}\notag\\
 		&\quad \quad \quad \quad \quad\!-p\left|X^{i,N}(s)\right| ^{p-2}\left\langle X^{i,N}(s),\gamma(X^{i,N}(\check{s}),\mu_{\check{s}}^{X,N},z)\right\rangle \Big) N^i(\mathrm{d}s,\mathrm{d}z)\Bigg)\Bigg] \notag\\
 		\leq &\mathbb{E}\left| \xi^i\right|^{p}+\frac{1}{2}\mathbb{E}\left(\sup_{0 \leq t \leq T_1}\left|X^{i,N}(t) \right|^{p}  \right)\notag\\
 		&+C\mathbb{E}\left[  \int_{0}^{T_1}\left(\left|X^{i,N}(t) \right|^{p}+\left|X^{i,N}(\check{t}) \right|^{p} +\mathbb{W}_2^p(\mu_{t}^{X,N},\delta_0)  +\mathbb{W}_2^p(\mu_{\check{t}}^{X,N},\delta_0) \right) \mathrm{d}t\right] \notag\\
 		&+\!p\mathbb{E}\left[  \sup_{0 \leq t \leq T_1}\left( \int_{0}^{t}\left| X^{i,N}(s)\right| ^{p-2}\left\langle X^{i,N}(s)\!-X^{i,N}(\check{s}),b_{\Delta}(X^{i,N}(\check{s}),\mu_{\check{s}}^{X,N})\right\rangle \mathrm{d}s\right)\right].
 	\end{align}
 	Using \eqref{eq3.1}, the Cauchy-Schwarz inequality,  Young's inequality, Kunita's inequality, H{\"o}lder's inequality  and Fubini theorem, it holds that
 	\begin{align}
 		&p\mathbb{E}\!\left[\sup_{0\le t\le T_1}\left(\int_0^t\!
 		|X^{i,N}(s)|^{p-2}
 		\left\langle X^{i,N}(s)\!-X^{i,N}(\check{s}),
 		b_{\Delta}(X^{i,N}(\check{s}),\mu_{\check{s}}^{X,N})
 		\right\rangle \mathrm{d}s\right)\right] \notag\\[2pt]			
 		\leq& p\mathbb{E}\left( \int_0^{T_1}
 		|X^{i,N}(t)|^{p-2}
 		\left|\int_{\check{t}}^{t}\!
 		b_{\Delta}(X^{i,N}(\check{s}),\mu_{\check{s}}^{X,N})\,\mathrm{d}s\right|
 		\left|b_{\Delta}(X^{i,N}(\check{t}),\mu_{\check{t}}^{X,N})\right|\mathrm{d}t\right)  \notag\\
 		&\!+p\mathbb{E}\left( \int_0^{T_1}\!
 		|X^{i,N}(t)|^{p-2}
 		\left|\int_{\check{t}}^{t}
 		g(X^{i,N}(\check{s}),\mu_{\check{s}}^{X,N})\,\mathrm{d}B^i(s)\right|
 		\left|b_{\Delta}(X^{i,N}(\check{t}),\mu_{\check{t}}^{X,N})\right|\mathrm{d}t\right)  \notag\\
 		&\! +p\mathbb{E}\left( \int_0^{T_1}\!
 		|X^{i,N}(t)|^{p-2}
 		\left|\int_{\check t}^{t}\!\int_Z
 		\gamma(X^{i,N}(\check{s}),\mu_{\check{s}}^{X,N},z)
 		\widetilde N^i(\mathrm{d}s,\mathrm{d}z)\right|
 		\left|b_{\Delta}(X^{i,N}(\check{t}),\mu_{\check{t}}^{X,N})\right|\mathrm{d}t\right)  \notag\\[6pt]
 		\leq{}&
 		C\int_0^{T_1}\mathbb{E}|X^{i,N}(t)|^p\,\mathrm{d}t
 		+C\int_0^{T_1}\mathbb{E}\left(1+\left|X^{i,N}(\check t)\right|
 		+\mathbb{W}_2(\mu^{X,N}_{\check t},\delta_0) \right) ^p\,\mathrm{d}t \notag                                     \\
 		&+C\int_0^{T_1}
 		\mathbb{E}\left[
 		\left|g(X^{i,N}(\check t),\mu_{\check t}^{X,N})\right|^{\frac p2}
 		\left(1+\left|X^{i,N}(\check t)\right|
 		+\mathbb{W}_2(\mu^{X,N}_{\check t},\delta_0) \right)^{\frac p2}
 		\right]\mathrm{d}t                                                \notag        \\
 		&+C\int_0^{T_1}
 		\mathbb{E}\left[
 		\left(
 		\int_Z
 		\left|\gamma(X^{i,N}(\check t),\mu_{\check t}^{X,N},z)\right|^2
 		\nu(\mathrm{d}z)
 		\right)^{\frac p4}
 		\left(1+\left|X^{i,N}(\check t)\right|
 		+\mathbb{W}_2(\mu^{X,N}_{\check t},\delta_0) \right)^{\frac p2}
 		\right]\mathrm{d}t ,\notag
 	\end{align}
 	where by the definition of \(b_\Delta\) and Assumption \ref{dos},
 	we have
 	\begin{align}\label{b}
 		\left|b_\Delta(x,\mu)\right|
 		\le C\Delta^{-\frac12}\bigl(1+|x|+\mathbb{W}_2(\mu,\delta_0)\bigr).
 	\end{align}
 	\par Since $ X^{i,N}(t),i=1,2,\cdots,N $ are identically distributed, substituting above inequality into (\ref{eq3.4}), using \eqref{5.2} and Young's inequality yields
 	\begin{align}
 		\mathbb{E}\left(\sup_{0 \leq t \leq T_1}\left|X^{i,N}(t)\right|^{p}\right)
 		\leq C\Bigg[
 		1 + \int_{0}^{T_1}\mathbb{E}\left(\sup_{0 \leq s \leq t}\left|X^{i,N}(s)\right|^{p}\right) \mathrm{d}t
 		\Bigg].\notag
 	\end{align}
 	Thanks to the Gr{\"o}nwall inequality, we obtain the desired assertion.
 \end{proof}
 \begin{lemma}\label{3.2}
	Let  Assumptions \ref{dos},\ref{p} and \ref{GAMMA} hold. Then, there exists a constant $C>0$, independent of $N$ and $\Delta$, such that
	\begin{eqnarray*}
		\sup_{1 \leq i \leq N}\sup_{0 \leq t \leq T}\mathbb{E}\left| X^{i,N}(t)-X^{i,N}(\check{t})\right| ^{p} \leq C\Delta \quad \text{for any $p\in\left[ 2,p_0\right]$,}
	\end{eqnarray*}
	where $p_0$ is given in Assumption \ref{p}.
\end{lemma}
\begin{proof}
	It directly follows from scheme (\ref{eq3.1}), estimate	\eqref{b}, H{\"o}lder's inequality, Kunita's inequality, Assumptions \ref{GAMMA} and Lemma \ref{3.1} that
	\begin{align}
		&\mathbb{E}\left| X^{i,N}(t)-X^{i,N}(\check{t})\right| ^{p}\notag\\
		\leq &C\mathbb{E}\left|
		\int_{\check t}^{t}
		b_\Delta\bigl(X^{i,N}(\check{s}),\mu_{\check{s}}^{X,N}\bigr)
		\,\mathrm{d}s
		\right|^p +C\mathbb{E}\left| g(X^{i,N}(\check{t}),\mu_{\check{t}}^{X,N}) \left( B^i(t)\!-B^i(\check{t})\right) \right|^{p}\notag\\
		&+C \mathbb{E} \left(\int_{\check{t}}^{t}\int_{Z}\!\left| \gamma(X^{i,N}(\check{s}),\mu_{\check{s}}^{X,N},z)\right| ^{2}\nu(\mathrm{d}z)\mathrm{d}s\right) ^\frac{p}{2}+C \mathbb{E} \left(\int_{\check{t}}^{t}\int_{Z}\!\left| \gamma(X^{i,N}(\check{s}),\mu_{\check{s}}^{X,N},z)\right| ^{p}\nu(\mathrm{d}z)\mathrm{d}s\right) \notag\\
		\leq &	C\Delta^{\frac p2}
		\mathbb{E}\left[
		1+\left|X^{i,N}(\check t)\right|^p
		+\mathbb{W}_2^p(\mu_{\check t}^{X,N},\delta_0)
		\right] +C\Delta^{\frac{p}{2}}\mathbb{E}\left| g(X^{i,N}(\check{t}),\mu_{\check{t}}^{X,N})\right| ^{p}  \notag\\
		&+C\Delta^{\frac{p}{2}}\mathbb{E} \left(\int_{Z}\left| \gamma(X^{i,N}(\check{t}),\mu_{\check{t}}^{X,N},z)\right|^{p} \nu(\mathrm{d}z)\right)+C\Delta\mathbb{E} \left(\int_{Z}\left| \gamma(X^{i,N}(\check{t}),\mu_{\check{t}}^{X,N},z)\right|^{p} \nu(\mathrm{d}z)\right) \notag\\
		\leq &C\Delta.\notag
	\end{align}
	The proof is complete.
\end{proof}
 \par According to the IPS \eqref{eq2.7} and scheme \eqref{eq3.1}, let $ e^i(t):=x^{i,N}(t) -X^{i,N}(t) $ denote the error. The following theorem will illustrate the convergence order of time discretization.
 \begin{theorem}\label{3.3}
	Let Assumptions \ref{g},\ref{dos},\ref{p} and \ref{GAMMA} hold with $2(2q+1)\vee \frac{2q(2+\delta)}{\delta}\le p_0$, $\delta\in(0,1)$. Then, there exists a constant $C>0$, independent of $N$ and $\Delta $, such that
	\begin{eqnarray*}
		\sup_{1 \leq i \leq N}\mathbb{E}\left( \sup_{0 \leq t \leq T}\left|x^{i,N}(t)-X^{i,N}(t) \right|^{2} \right) \leq C\Delta^{\frac{2}{2+\delta}},
	\end{eqnarray*}
	where $q$ is given in Assumption \ref{dos} and $p_0$ is given in Assumption \ref{p}.
\end{theorem}
\begin{proof}
	Applying It{\^o}'s formula, we have
	\begin{align}
		\left|e^i(t) \right|^{2}
		\leq& 2\int_{0}^{t} \left\langle e^i(s) ,b(x^{i,N}(s),\mu_s^{x,N})\!-b_{\Delta}(X^{i,N}(\check{s}),\mu_{\check{s}}^{X,N}) \right\rangle \mathrm{d}s\notag\\
		&+\!2\int_{0}^{t}\left\langle e^i(s) ,g(x^{i,N}(s),\mu_s^{x,N})\!-g(X^{i,N}(\check{s}),\mu_{\check{s}}^{X,N}) \right\rangle \mathrm{d}B^i(s)\notag\\
		&+\int_{0}^{t} \left|  g(x^{i,N}(s),\mu_s^{x,N}) \!-g(X^{i,N}(\check{s}),\mu_{\check{s}}^{X,N}) \right|^2\mathrm{d}s\notag\\
		&+2\int_{0}^{t}\int_{Z} \left\langle e^i(s) ,\gamma(x^{i,N}(s),\mu_{s}^{X,N},z)\!-\gamma(X^{i,N}(\check{s}),\mu_{\check{s}}^{X,N},z) \right\rangle\widetilde{N}^i(\mathrm{d}s,\mathrm{d}z)\notag\\
		&+\int_{0}^{t}\int_{Z} \left|\gamma(x^{i,N}(s),\mu_{s}^{x,N},z)-\gamma(X^{i,N}(\check{s}),\mu_{\check{s}}^{X,N},z)\right|^{2} N^i(\mathrm{d}s,\mathrm{d}z).\notag
	\end{align}
	By using Assumption \ref{dos}, Remark \ref{remark} and the definition of  $ b_{\Delta} $, one obtains
	\begin{align}
		&\langle e^i(s), b(x^{i,N}(s),\mu_s^{x,N}) - b_{\Delta}(X^{i,N}(\check{s}),\mu_{\check{s}}^{X,N}) \rangle \notag \\
		\leq& \langle e^i(s), b(x^{i,N}(s),\mu_s^{x,N}) - b(X^{i,N}(s),\mu_s^{X,N}) \rangle \notag \\
		& + \langle e^i(s), b(X^{i,N}(s),\mu_s^{X,N}) - b(X^{i,N}(\check{s}),\mu_{\check{s}}^{X,N}) \rangle  \notag\\
		&+ \langle e^i(s), b(X^{i,N}(\check{s}),\mu_{\check{s}}^{X,N}) - b_{\Delta}(X^{i,N}(\check{s}),\mu_{\check{s}}^{X,N}) \rangle \notag \\
		\leq{}&
		C|e^i(s)|^2
		+C\mathbb W_2^2(\mu_s^{x,N},\mu_s^{X,N})+
		C|X^{i,N}(s)-X^{i,N}(\check{s})|^2
		\left(1+|X^{i,N}(s)|^{2q}+|X^{i,N}(\check{s})|^{2q}\right) \notag\\
		&+
		C\mathbb W_2^2(\mu_s^{X,N},\mu_{\check{s}}^{X,N})+C\sqrt\Delta |e^i(s)|
		\left(1+|X^{i,N}(\check{s})|^{2q+1}\right),\notag
	\end{align}
where	$\left| b(x,\mu)-b_\Delta(x,\mu)\right| \leq \left|b_1(x) \right|\left| x\right|^q \sqrt{\Delta}$. By Assumption \ref{g}, we have
	\begin{align}
		&|g(x^{i,N}(s),\mu_s^{x,N}) - g(X^{i,N}(\check{s}),\mu_{\check{s}}^{X,N})|^2 \notag \\
		\leq &2|g(x^{i,N}(s),\mu_s^{x,N}) - g(X^{i,N}(s),\mu_s^{X,N})|^2  + 2|g(X^{i,N}(s),\mu_s^{X,N}) - g(X^{i,N}(\check{s}),\mu_{\check{s}}^{X,N})|^2 \notag \\
		\leq &C\Big(|e^i(s)|^2 +\! \mathbb{W}_2^2(\mu_s^{x,N},\mu_s^{X,N})\Big) + C\Big(|X^{i,N}(s) - X^{i,N}(\check{s})|^2+ \mathbb{W}_2^2(\mu_s^{X,N},\mu_{\check{s}}^{X,N}) \Big),\notag
	\end{align}
	and 
	\begin{align}
		&\int_{Z}\left|\gamma(x^{i,N}(t),\mu_t^{x,N},z)
		-\gamma(X^{i,N}(\check t),\mu_{\check t}^{X,N},z) \right|^{2} \nu(\mathrm{d}z)\notag \\
		\leq &2\bigg(\int_{Z}\left|\gamma(x^{i,N}(t),\mu_t^{x,N},z)-\gamma(X^{i,N}(t),\mu_t^{X,N},z) \right|^{2}\nu(\mathrm{d}z) \notag\\
		&\qquad +\int_{Z} \left|\gamma(X^{i,N}(t),\mu_t^{X,N},z)-\gamma(X^{i,N}(\check{t}),\mu_{\check t}^{X,N},z) \right|^{2} \nu(\mathrm{d}z) \bigg)\notag \\
		\leq &C\Big(|e^i(t)|^2 +|X^{i,N}(t) - X^{i,N}(\check{t})|^2+\mathbb W_2^2(\mu_t^{x,N},\mu_t^{X,N})
		+\mathbb W_2^2(\mu_t^{X,N},\mu_{\check t}^{X,N})\Big).\notag
	\end{align}
	Thus, for any $ T_1\in[0,T] $, it follows from Young's inequality, H{\"o}lder's inequality, the B-D-G inequality and Kunita's inequality that 
	\begin{align}
		&\mathbb{E}\left( \sup_{0 \leq t \leq T_1}\left|e^i(t) \right|^{2}\right)\notag\\
		\leq{}&
		C\mathbb{E}\int_{0}^{T_1}
		\Bigg[
		\left| e^i(t)\right|^{2}
		+\mathbb{W}_2^2(\mu_t^{x,N},\mu_{t}^{X,N})
		+\mathbb{W}_2^2(\mu_t^{X,N},\mu_{\check{t}}^{X,N})
		\notag\\
		&\qquad
		+\Delta
		\left(1+\left|X^{i,N}(\check t)\right|^{2q+1}\right)^2
		+\left|X^{i,N}(t)-X^{i,N}(\check t)\right|^2
		\notag\\
		&\qquad
		+\left|X^{i,N}(t)-X^{i,N}(\check t)\right|^{2}
		\left(
		1+\left|X^{i,N}(t)\right|^{2q}
		+\left|X^{i,N}(\check t)\right|^{2q}
		\right)
		\Bigg]\mathrm{d}t
		\notag\\
		&+C\mathbb{E}\left(
		\int_{0}^{T_1}
		\left| e^i(t)\right|^{2}
		\left|
		g(x^{i,N}(t),\mu_t^{x,N})
		-g(X^{i,N}(\check t),\mu_{\check t}^{X,N})
		\right|^2
		\mathrm{d}t
		\right)^{\frac12}
		\notag\\
		&+C\mathbb{E}\left(
		\int_{0}^{T_1}\int_Z
		\left| e^i(t)\right|^{2}
		\left|
		\gamma(x^{i,N}(t),\mu_t^{x,N},z)
		-\gamma(X^{i,N}(\check t),\mu_{\check t}^{X,N},z)
		\right|^2
		\nu(\mathrm{d}z)\mathrm{d}t
		\right)^{\frac12}
		\notag\\
		&+C\mathbb{E}\int_{0}^{T_1}\int_Z
		\left|
		\gamma(x^{i,N}(t),\mu_t^{x,N},z)
		-\gamma(X^{i,N}(\check t),\mu_{\check t}^{X,N},z)
		\right|^2
		\nu(\mathrm{d}z)\mathrm{d}t\notag \\
		\leq{}&
		\frac12\mathbb{E}\left( \sup_{0 \leq t \leq T_1}\left|e^i(t) \right|^{2}\right)
		+
		C\mathbb{E}\int_{0}^{T_1}
		\Big[
		|e^i(t)|^2
		+\mathbb W_2^2(\mu_t^{x,N},\mu_t^{X,N})
		+\mathbb W_2^2(\mu_t^{X,N},\mu_{\check t}^{X,N})
		\notag\\
		&\qquad
		+|X^{i,N}(t)-X^{i,N}(\check t)|^2
		+|X^{i,N}(t)-X^{i,N}(\check t)|^2
		\left(
		1+|X^{i,N}(t)|^q
		+|X^{i,N}(\check t)|^q
		\right)^2
		\notag\\
		&\qquad
		+\Delta
		\left(1+|X^{i,N}(\check t)|^{2q+1}\right)^2
		\Big]\mathrm{d}t.\notag
	\end{align}
	
	 Since $X^{i,N}(t) $ for $i=1,2,\cdots,N$ are identically distributed, combining Lemma \ref{3.1} and \ref{3.2}, it can be deduced from the Fubini theorem, H{\"o}lder's inequality and Young's inequality that for any $\delta\in(0,1)$,
	\begin{align}
		&\mathbb{E}\left( \sup_{0\le t\le T_1}\bigl|e^i(t)\bigr|^2\right)\notag\\
		\le& C\int_{0}^{T_1}\bigg[
		\mathbb{E}\Bigl(\sup_{0\le s\le t}\bigl|e^i(s)\bigr|^2\Bigr)+\mathbb{E}\Bigl(\frac1N\sum_{j=1}^{N}\bigl|e^j(t)\bigr|^2\Bigr)\!+\Delta+\mathbb{E}\Bigl(\frac1N\sum_{j=1}^{N}\bigl|X^{j,N}(t)-X^{j,N}(\check{t})\bigr|^2\Bigr)	\notag\\
		&+\Bigl(\mathbb{E}\bigl|X^{i,N}(t)-X^{i,N}(\check{t})\bigr|^{2+\delta}\Bigr)^{\frac{2}{2+\delta}}\Bigl(\mathbb{E}\left( 1+\left|X^{i,N}(t) \right|^q+\left|X^{i,N}(\check{t}) \right|^q \right)^{\frac{2(2+\delta)}{\delta}}\Bigr)^{\frac{\delta}{2+\delta}}
		\bigg] \mathrm{d}t \notag\\[4pt]
		\le& C\int_{0}^{T_1}\mathbb{E}\Bigl(\sup_{0\le s\le t}\bigl|e^i(s)\bigr|^2\Bigr)\mathrm{d}t + C\left( \Delta+ \Delta+ \Delta^{\frac{2}{2+\delta}}\right) .\notag
	\end{align}
	By Gr{\"o}nwall's inequality, we get the desired assertion.
\end{proof}
 Combining Theorem \ref{2.6} and \ref{3.3}, we get the error estimate of MV-SDEs \eqref{eq1.1} driven by L{\'e}vy noise in finite time.
 \begin{corollary}\label{3.4}
 	Let Assumptions \ref{g},\ref{dos},\ref{p} and \ref{GAMMA} hold with $p_0>4$ and $2(2q+1)\vee \frac{2q(2+\delta)}{\delta}\le p_0$, $\delta\in(0,1)$. Then, there exists a constant $C>0$, independent of $N$ and $\Delta $, such that
 	\begin{eqnarray*}
 		\sup_{1 \leq i \leq N}\!\mathbb{E}\left( \sup_{0 \leq t \leq T}\left|x^i(t)-X^{i,N}(t) \right|^{2} \right) \leq C\left(  \varphi(N)+\Delta^{\frac{2}{2+\delta}}\right).
 	\end{eqnarray*}
 \end{corollary}

  In the end, we prove that the numerical invariant measure converges to the underlying one. Recall that $\Pi^{\Delta,N}$ is the invariant measure of the tamed Euler scheme on $(\mathbb{R}^{d})^N$, whereas $\mu^*\in\mathcal P_2(\mathbb R^d)$ is the invariant measure of the MV-SDE. Since these two measures are defined on different spaces, we compare $\Pi^{\Delta,N}$ with the tensor product measure $(\mu^*)^{\otimes N}$ on $(\mathbb{R}^{d})^N$. Let $\mathbf{P}_{t}\mu^{\otimes N}$
  denote the law of $(x^1(t),\cdots,x^N(t))$ for the NIPS \eqref{eq2.8} with initial law \(\mu^{\otimes N}\)
  . By tensorizing the contraction estimate of Theorem \ref{4.2}, and using an 
  argument analogous to that in Theorem \ref{4.5}, we obtain the following 
  corollary.
  \begin{corollary}\label{joint}
  	Let Assumptions \ref{2.1}-\ref{g} hold with $2\lambda_1-2\lambda_2-8\lambda_{3}-1>0$. Then, \begin{align*}
  		\mathbb{W}_{2}^{2}\left( \mathbf{P}_{t}\mu^{\otimes N},\left( \mu^*\right)^{\otimes N} \right) \leq e^{ -\lambda^* t}\mathbb{W}_{2}^{2}(\mu^{\otimes N},\left( \mu^*\right)^{\otimes N}),\ \ t\geq 0,\,\mu\in\mathcal{P}_{2}(\mathbb{R}^d).
  	\end{align*}
  \end{corollary}
  According to \cite{soni2025}, particles are exchangeable and thus $\mathcal{L}(X^{i,N}(t))=\mathcal{L}(X^{j,N}(t))$ for any $t,i,j$. Therefore, by considering the marginal distribution of any particle, denote by $\mu^{\Delta,N}\in\mathcal{P}_2(\mathbb{R}^d)$, we derive the convergence between the numerical invariant measure $\mu^{\Delta,N} $ and the exact invariant measure $\mu^*$. 
 \begin{theorem}	\label{convergence}
 Let Assumptions \ref{g}, \ref{dos}, \ref{p}   and \ref{GAMMA} hold with $(L_1\land 2L_3) -1-6L_5-2(m+1)\lambda_3>0$, $p_0>4$
 and $2(2q+1)\vee \frac{2q(2+\delta)}{\delta}\le p_0$, $\delta\in(0,1)$. Then
 	\[
 	\lim_{\substack{N\to\infty\\ \Delta\to0}}
 	\mathbb{W}_2^2(\mu^{\Delta,N},\mu^*)=0.
 	\]
 \end{theorem}
 \begin{proof}
 	For any \( t > 0 \) and \( \Delta \in (0, \Delta^{*}] \),
 	take \( k \in \mathbb{N}_0 \) such that \( t \in [t_k, t_{k+1}) \). Since \( X^{i, N}(t) = X^{i, N}(t_k) \) for \( t \in [t_k, t_{k+1}) \), using Theorem \ref{4.5} we obtain
 	\begin{align}\label{Pt}
 		\mathbb{W}_{2}^2\bigl(\mathbf{P}^{\Delta,N}_{t}\mu^{\otimes N},\,\Pi^{\Delta,N}\bigr) = \mathbb{W}_{2}^2\bigl(\mathbf{P}^{\Delta,N}_{t_k}\mu^{\otimes N},\,\Pi^{\Delta,N}\bigr)
 		\leq e^{-\bar{\lambda}^* t_k}\mathbb{W}_{2}^2\bigl(\mu^{\otimes N},\Pi^{\Delta,N}\bigr).
 	\end{align}
 	Then, by Lemma \ref{new1} we have
 	\begin{align*}
 		\frac{1}{N}\mathbb{W}_{2}^{2}\left(\mathbf{P}^{\Delta,N}_{t}\mu^{\otimes N},\,\Pi^{\Delta,N} \right) \leq Ce^{ -\bar{\lambda}^* t_k}.
 	\end{align*}
 	 Based on the inequality above and Corollary \ref{joint}, for any \(\varepsilon>0\), there exists a sufficiently large constant \(T>0\) such that
 	\[
 	\frac1N\mathbb{W}^{2}_{2}\left(\mathbf{P}_{T}\mu^{\otimes N},(\mu^*)^{\otimes N}\right)<\frac{\varepsilon}{9}, \quad \frac1N\mathbb{W}^{2}_{2}\left(\Pi^{\Delta,N},\mathbf{P}_T^{\Delta,N}\mu^{\otimes N}\right)<\frac{\varepsilon}{9}.
 	\]
 	Moreover, for $\xi^i\sim \mu$, Corollary   \ref{3.4} gives
 	\begin{align*}
 		\lim_{\substack{N\to\infty\\ \Delta\to0}}\frac1N\mathbb{W}^{2}_{2}\left(\mathbf{P}_T^{\Delta,N}\mu^{\otimes N},\mathbf{P}_{T}\mu^{\otimes N}\right) \leq &\lim_{\substack{N\to\infty\\ \Delta\to0}}\frac1N\sum_{i=1}^N
 		\mathbb E\left|X^{i,N}(T)-x^i(T)\right|^2\\
 		\leq &\lim_{\substack{N\to\infty\\ \Delta\to0}}
 		\sup_{1\le i\le N}
 		\mathbb E\left(
 		\sup_{0\le t\le T}|X^{i,N}(t)-x^i(t)|^2
 		\right), 
 	\end{align*}
 	where \(x^{i}_T\) and \(X^{i,N}_T\), \(i=1,2,\cdots,N\), are the exact solution of equation \eqref{eq2.8} and the numerical  solution generated by the tamed Euler scheme \eqref{discrete} with the same initial value $\xi^i$, respectively. Hence, for any fixed \(\varepsilon>0\), we can  select an integer \(N^{*}\) sufficiently large and \(\Delta_1\in(0,\Delta^{*})\) sufficiently small such that for all \(N>N^{*}\) and \(\Delta\in(0,\Delta_1)\),
 	\[
 	\frac1N\mathbb{W}^{2}_{2}\left(\mathbf{P}_T^{\Delta,N}\mu^{\otimes N},\mathbf{P}_{T}\mu^{\otimes N}\right) < \frac{\varepsilon}{9}.
 	\]
 	Thus,  applying the triangle  inequality, we obtain for any \(N>N^{*}\) and \(\Delta\in(0,\Delta_1)\),
 \[
 \begin{aligned}
 	\frac1N\mathbb{W}_2^2\left(\Pi^{\Delta,N},(\mu^*)^{\otimes N}\right) \le&
 	\frac3N\mathbb{W}_2^2\left(\Pi^{\Delta,N},\mathbf{P}_T^{\Delta,N}\mu^{\otimes N}\right)
 	+\frac3N\mathbb{W}_2^2\left(\mathbf{P}_T^{\Delta,N}\mu^{\otimes N},\mathbf{P}_{T}\mu^{\otimes N}\right) \\
 	&+\frac3N\mathbb{W}_2^2\left(\mathbf{P}_{T}\mu^{\otimes N},(\mu^*)^{\otimes N}\right)\\
 	<&\varepsilon.
 \end{aligned}
 \]

 		According to \cite{soni2025}, since the IPS forms an exchangeable system across time and the tamed Euler scheme is the discretization of IPS, $\Pi^{\Delta,N}$ is exchangeable and therefore all its one-particle marginals are identical. Denote this common marginal by $\mu^{\Delta,N}\in\mathcal{P}_2(\mathbb{R}^d)$, then by the following estimate
 	\begin{align}\label{key}
 		\mathbb{W}_2^2\bigl(\mu^{\Delta,N},\;\mu^*\bigr)
 		\leq& \frac{1}{N}\,	\mathbb{W}^2_{2}(\Pi^{\Delta,N},(\mu^*)^{\otimes N}),
 	\end{align}
 	the proof is complete.
 \end{proof}
 \section{Convergence rate of numerical invariant measure}
Section 5 establishes qualitative convergence of numerical invariant measure, while in Section 6, we strengthen this result by deriving an explicit convergence rate via uniform-in-time estimates. 
\subsection{Uniform-in-time PoC}
 
\begin{lemma}\label{uniform}
	Let Assumptions \ref{dos}, \ref{p}  and \ref{GAMMA} hold. Then 
	\[
	\sup_{t\ge 0}\mathbb E |x(t)|^{p}
	\le C \quad \text{for any}~p\in[2,p_0],
	\]
	where \(p_0\) is given in Assumption \ref{p}.
\end{lemma}
 \begin{proof}
 	To begin with, by Assumption \ref{dos} and Young's inequality, we get
 	\[
 	\begin{aligned}
 		\langle x,b(x,\mu)\rangle
 		&\le -L_1|x|^{q+2}+C(1+|x|^2+\mathbb{W}_2^2(\mu,\delta_0)).
 	\end{aligned}
 	\]
 	
 	 Similar as the proof in Lemma \ref{2.5}, by applying
 	Itô's formula to \(|x(t)|^{p}\), using remainder formula \eqref{eq2.5}, Young's inequality and Assumption \ref{GAMMA},
 	we obtain
 	\begin{align*}
 		\mathbb E|x(t)|^{p}
 		\leq&\mathbb{E}\left| \xi\right|^p+ \int_{0}^{t}\left(  -pL_1\mathbb E|x(s)|^{p+q}
 		+C\mathbb E|x(s)|^{p}\right) \mathrm{d}s+C.
 	\end{align*}
 	Since \(q>0\), there exist a constant \(K>0\) such that
 		\begin{align*}
 		\mathbb E|x(t)|^{p}
 		\leq&\mathbb{E}\left| \xi\right|^p-K \int_{0}^{t}\mathbb E|x(s)|^{p} \mathrm{d}s+C.
 	\end{align*}
 	The proof is complete with Gr{\"o}nwall's inequality.
 \end{proof}
 \begin{theorem}\label{uniform-in time PoC}
 	Let  Assumptions \ref{g}, \ref{dos}, \ref{p} and \ref{GAMMA} hold with $p_0>4$ and
 	$2L_{1}-6\lambda_3-1-3L_{5}>0$. Then, there exists a constant $C>0$, independent of $t$ and $N$, such that 
 	\begin{align*}
 		\sup_{1 \leq i \leq N}\sup_{t\geq 0}\mathbb{E}\left|x^i(t)-x^{i,N}(t) \right|^2\leq C\varphi(N),
 	\end{align*}
 	where  $\varphi(N)$ is given in Theorem \ref{2.6}.
 	
 \end{theorem}
 \begin{proof}
 	Let $ e^{i,N}(t):=x^i(t)-x^{i,N}(t) $ for any $t\geq 0$. Using It{\^o}'s formula, Remark \ref{remark}  and Assumption \ref{g}, we get for $\lambda>0$,
 	\begin{align*}
 		&\mathbb{E}\left( 	e^{\lambda t}|e^{i,N}(t)|^{2}\right) \\
 		\leq& \mathbb{E}\bigg[ \int_{0}^{t} e^{\lambda s}\left((\lambda-2L_{1}+2\lambda_3+1+L_{5}) |e^{i,N}(s)|^{2}+(L_{5}+2\lambda_3)\mathbb{W}^{2}_{2}(\mu_s^{x^i},\mu_s^{x,N})\right) \mathrm{d}s\bigg] \\
 		\leq &\int_{0}^{t} e^{\lambda s}\left[ (\lambda-2L_{1}+6\lambda_3+1+3L_{5}) \mathbb{E} |e^{i,N}(s)|^{2}+2(L_{5}+2\lambda_3)\mathbb{E}\mathbb{W}^{2}_{2}(\mu_s^{x^i},\mu_s^{N})\right]  \mathrm{d}s,
 	\end{align*}
 	where we also apply \eqref{w2} to the final estimate, $\mu_s^{N}=\frac{1}{N}\sum_{j=1}^{N}\delta_{x^{j}(s)} $. Choose  $\lambda=2L_{1}-6\lambda_3-1-3L_{5}>0$, then it follows that
 	\begin{align*}
 		\mathbb{E} |e^{i,N}(t)|^{2}
 		\leq 2(L_{5}+2\lambda_3)\int_{0}^{t} e^{-\lambda (t-s)} \mathbb{E}\mathbb{W}^{2}_{2}(\mu_s^{x^i},\mu_s^{N})\mathrm{d}s.
 	\end{align*}
 	Therefore, by the standard empirical measure estimate in \cite{CarmonaI2018} and  Lemma \ref{uniform}, we can get the desired assertion.
 \end{proof}
 \subsection{Uniform-in-time convergence of tamed Euler method}
 In order to prove the uniform
 moment boundedness of the numerical solution, we follow the strategy of \cite{BAO2025} and need the additional assumption as follows.
 \begin{assumption}\label{bound}
 	There exists a constant $M_0>0$ such that for any $ x\in \mathbb{R}^d $ and $
 	\mu \in \mathcal P_2(\mathbb{R}^d)  $,
 	\begin{align*}
 		\left|g(x,\mu) \right|^2\vee \int_{Z}\left|\gamma(x,\mu,z) \right|^2\nu(\mathrm{d}z)\leq M_0.
 	\end{align*}
 \end{assumption}
 
 \begin{lemma}\label{new}
 	Let Assumptions \ref{dos},\ref{p}  and \ref{bound} hold and suppose $\kappa:=2\left( L_1-1- L_5\right) >0$.
 	Then, for any $\Delta\in (0,\Delta^{**}] $, $\Delta^{**}:= 1\land \frac{\kappa^2}{4\left(24L_0+2+10L_5\right)^2 } \land \frac{1}{\frac{\kappa}{4} + 6L_5 },$
 	\begin{align*}
 		\sup_{k\in \mathbb{N}_0}\mathbb{E}\left|  X^{i,N}(t_{k})\right|^{p}\leq C \quad \text{for any $p\in\left[ 2,p_0\right]$,}
 	\end{align*}
 	where $p_0$ is given in Assumption \ref{p}.
 \end{lemma}
 \begin{proof}
 	From scheme \eqref{discrete} and equation \eqref{*},
 	using Assumption \ref{dos} gives
 	\begin{align*}
 		I_1\leq	& 2\left\langle X^{i,N}(t_{k}), \frac{b_1\left( X^{i,N}(t_{k})\right) }{1+\sqrt{\Delta}\left|X^{i,N}(t_{k}) \right|^q}+b_2\left(X^{i,N}(t_{k}),\mu_{t_k}^{X,N}\right)  \right\rangle\\
 		&+ \Delta \left| \frac{b_1\left( X^{i,N}(t_{k})\right) }{1+\sqrt{\Delta}\left|X^{i,N}(t_{k}) \right|^q}+b_2\left(X^{i,N}(t_{k}), \mu_{t_k}^{X,N}\right) \right| ^2\\
 		\leq &\frac{\left|X^{i,N}(t_{k}) \right|^{2}}{1+\sqrt{\Delta}\left|X^{i,N}(t_{k}) \right|^q}\Bigg(-2L_1\left( 1+\left|X^{i,N}(t_{k}) \right|^{q}\right) + \frac{8\Delta L_0\left( 1\!+\left|X^{i,N}(t_{k}) \right|^{2q}\right)  }{ 1\!+\sqrt{\Delta}\left|X^{i,N}(t_{k}) \right|^q}\Bigg)\\
 		&+2\left|X^{i,N}(t_{k}) \right|^{2} +2L_{5}(1+2\Delta) \left(\left|X^{i,N}(t_{k}) \right|^{2}+ \mathbb{W}_2^2(\mu_{t_k}^{X,N},\delta_0)\right)\\
 		&+(1+4\Delta)\left|b_1(0) \right|^2 +\!2(1+2\Delta)\left|b_2(0,\delta_0) \right|^2  \\
 		\leq &\frac{-\left|X^{i,N}(t_{k}) \right|^q}{1+\sqrt{\Delta}\left|X^{i,N}(t_{k}) \right|^q}\left(2L_1\!-8L_{0}\sqrt{\Delta} \right)\left|X^{i,N}(t_{k}) \right|^2+ 2L_{5}(1+2\Delta)\mathbb{W}_2^2(\mu_{t_k}^{X,N},\delta_0)\\
 		&+2(4L_{0}\Delta+\!1+L_{5}(1+2\Delta))\left|X^{i,N}(t_{k}) \right|^2+C.
 	\end{align*}
 	Thus, for any $\Delta\in (0,\Delta^{**}] $,
 	\begin{align*}
 		I_1\Delta 
 		\leq &\bigg[ \frac{-1}{1+\sqrt{\Delta}}\left(2L_1-8L_{0}\sqrt{\Delta}\right)\left|X^{i,N}(t_{k}) \right|^2+2(4L_{0}\Delta+\!1+L_{5}(1+2\Delta))\left|X^{i,N}(t_{k}) \right|^2\\
 		&+ 2L_{5}(1+2\Delta)\mathbb{W}_2^2(\mu_{t_k}^{X,N},\delta_0)\bigg]\Delta+ C_1\Delta\\
 		\leq &\bigg[ \frac{-1}{1+\sqrt{\Delta}}\left(\kappa-\sqrt{\Delta}\left( 24L_{0}+2+10 L_5\right)  \right)\left|X^{i,N}(t_{k}) \right|^2\\
 		&+2L_{5}(1+2\Delta)\left( \mathbb{W}_2^2(\mu_{t_k}^{X,N},\delta_0)-\left|X^{i,N}(t_{k}) \right|^2\right) \bigg]\Delta+ C_1\Delta
 	\end{align*}
 	for some constant $C_1>0$, where $\kappa=2\left( L_1-1- L_5\right)  >0$. For any $\Delta\in (0,\Delta^{**}]$, 
 	\begin{align*}
 		\sqrt{\Delta}\left(24L_{0}+2+10 L_5 \right)\leq \frac{\kappa}{2}.
 	\end{align*}
 	Owing to $\sqrt{\Delta}\leq1 $ for $\Delta\in (0,\Delta^{**}]$, we infer that
 	\begin{align*}
 		I_1\Delta
 		\leq &\bigg[ -\left( \frac{\kappa}{4}+\beta_\Delta\right)\left|X^{i,N}(t_{k}) \right|^2+\beta_\Delta\mathbb{W}_2^2(\mu_{t_k}^{X,N},\delta_0) \bigg]\Delta+ C_1\Delta,
 	\end{align*}
 	where  $\beta_\Delta:=2L_{5}(1+2\Delta) $. The above estimate enables us to deduce that
 	\begin{align*}
 		\left| X^{i,N}(t_{k+1}) \right|^2 \leq\left[1-\left( \frac{\kappa}{4}+\beta_\Delta\right)\Delta \right] \left| X^{i,N} (t_{k}) \right|^2 +\beta_\Delta\mathbb{W}_2^2(\mu_{t_k}^{X,N},\delta_0) \Delta+ C_1\Delta+I_2,
 	\end{align*}
 	where $ 1-\left( \frac{\kappa}{4}+\beta_\Delta\right)\Delta\geq0$ by taking $\Delta\in (0,\Delta^{**}]$ into consideration.
 	\par  Then	for any integer $p\geq 2$,
 	\begin{align*}
 		&\left| X^{i,N}(t_{k+1}) \right|^{2p} \\
 		\leq&\left\lbrace \left[1-\left( \frac{\kappa}{4}+\beta_\Delta\right)\Delta \right] \left| X^{i,N} (t_{k}) \right|^2 +\beta_\Delta\mathbb{W}_2^2(\mu_{t_k}^{X,N},\delta_0) \Delta\right\rbrace ^p\\
 		&+\!p\left\lbrace \left[1-\!\left( \frac{\kappa}{4}+\!\beta_\Delta\right)\Delta \right] \left| X^{i,N} (t_{k}) \right|^2\! +\beta_\Delta\mathbb{W}_2^2(\mu_{t_k}^{X,N},\delta_0) \Delta\right\rbrace ^{p-1}\!\left(  C_1\Delta\!+I_2\right) \\
 		&+\!\sum_{n=0}^{p-2}C^n_p\left\lbrace \left[1-\!\left( \frac{\kappa}{4}+\!\beta_\Delta\right)\Delta \right] \left| X^{i,N} (t_{k}) \right|^2 \!+\beta_\Delta\mathbb{W}_2^2(\mu_{t_k}^{X,N},\delta_0) \Delta\right\rbrace^n\!\left(  C_1\Delta+\!I_2\right)^{p-n}\\
 		=&:J_1+J_2+J_3.
 	\end{align*}
 	Given the \(\sigma\)-algebra \(\mathcal{F}^{N}_{0}\), which is generated by \(\xi^{1},\cdots,\xi^{N}\). Using the binomial theorem and the Young inequality yields that
 	\begin{align}\label{J1}
 		\mathbb{E}\left(J_1 \middle| \mathcal{F}^{N}_{0} \right)=&\mathbb{E}\left[  \sum_{n=0}^{p} C^{n}_{p} \big( 1 -\! (\kappa/4 +\! \beta_{\Delta}) \Delta \big)^{n} (\beta_{\Delta} \Delta)^{p-n} \left| X^{i,N}(t_{k}) \right|^{2n}\mathbb{W}_2^{2(p-n)} (\mu_{t_k}^{X,N},\delta_0)\middle| \mathcal{F}^{N}_{0}\right] \notag \\
 		\leq &\sum_{n=0}^{p} C^{n}_{p} \big( 1 - (\kappa/4 + \beta_{\Delta}) \Delta \big)^{n} (\beta_{\Delta} \Delta)^{p-n} \notag\\
 		&\times \left( \frac{n}{p} \mathbb{E} \left(\left| X^{i,N}(t_{k}) \right|^{2p} \middle|\mathcal{F}^{N}_{0} \right) + \frac{p-n}{p} \mathbb{E} \left( \mathbb{W}_2^{2p} (\mu_{t_k}^{X,N},\delta_0)\middle| \mathcal{F}^{N}_{0} \right) \right)\notag \\
 		= &(1 - \kappa \Delta / 4)^{p} \mathbb{E} \left( \left| X^{i,N}(t_{k}) \right|^{2p} \middle|\mathcal{F}^{N}_{0} \right) \notag\\
 		\leq& (1 - \kappa \Delta / 4) \mathbb{E} \left(\left| X^{i,N}(t_{k}) \right|^{2p} \middle| \mathcal{F}^{N}_{0} \right),
 	\end{align}
 	where in the second identity we used the fact that \(X^{i,N}(t_{k})\) and \(X^{j,N}(t_{k})\) are identically distributed given \(\mathcal{F}^{N}_{0}\), and the final estimate follows directly from \(1 - \kappa \Delta / 4 \in (0,1)\). Next, since 
 	\begin{equation}\label{delta}
 		\begin{split}
 			&\mathbb{E}\left( \triangle B^{i}_k\middle| \mathcal{F}^{N}_{0}\right)  = 0,\quad
 			\mathbb{E}\left( |\triangle B^{i}_k|^{2} \middle|\mathcal{F}^{N}_{0}\right)  =m \Delta,\\
 			&\mathbb{E}\left(\int_{t_{k}}^{t_{k+1}}\int_{Z} \gamma\left( X^{i,N}(t_{k}),\mu_{t_k}^{X,N},z\right) \widetilde{N}^i(\mathrm{d}s,\mathrm{d}z) \middle| \mathcal{F}^{N}_{0}\right)  =0,\\
 			&\mathbb{E}\left(\left| \int_{t_{k}}^{t_{k+1}}\int_{Z} \gamma\left( X^{i,N}(t_{k}),\mu_{t_k}^{X,N},z\right) \widetilde{N}^i(\mathrm{d}s,\mathrm{d}z)\right| ^2 \middle| \mathcal{F}^{N}_{0}\right) \\
 			=&\mathbb{E}\left( \int_{Z}\left| \gamma\left( X^{i,N}(t_{k}),\mu_{t_k}^{X,N},z\right)\right| ^2  \nu(\mathrm{d}z)\middle| \mathcal{F}^{N}_{0}\right)\Delta,
 		\end{split}
 	\end{equation}
 	it follows from Assumption \ref{bound} and Young's inequality that there is a constant \(C_{2} > 0\) such that
 	\begin{align}\label{J2}
 		\mathbb{E} \left(J_2\middle| \mathcal{F}^{N}_{0} \right) \leq& p C_{2} \Delta \, \mathbb{E} \left( \left( (1 -\! (\kappa/4 \!+ \beta_{\Delta}) \Delta) \left|  X^{i,N}(t_{k})\right|^{2} + \!\beta_{\Delta}\mathbb{W}_2^{2} (\mu_{t_k}^{X,N},\delta_0) \Delta \right)^{p-1} \middle| \mathcal{F}^{N}_{0} \right)\notag \\
 		\leq &\frac{\kappa \Delta}{16} \mathbb{E} \left( \left| X^{i,N}(t_{k})\right|^{2p} \middle| \mathcal{F}^{N}_{0} \right) + C_{2} \Delta. 
 	\end{align}
 	Using \eqref{delta} again and by Young’s inequality, we obtain
 	\begin{align}\label{J3}
 		\mathbb{E} \left(J_3\middle| \mathcal{F}^{N}_{0} \right)
 		\leq &\frac{\kappa \Delta}{16} \mathbb{E} \left( \left| X^{i,N}(t_{k})\right|^{2p} \middle| \mathcal{F}^{N}_{0} \right) + C_{3} \Delta
 	\end{align}
 	for some constant $C_{3}>0 $. Combining \eqref{J1}, \eqref{J2} and	\eqref{J3}, we conclude by induction that there exists a constant $C_1^*>0$ such that
 	\begin{align*}
 		\mathbb{E}\left( \left| X^{i,N}(t_{k+1})\right|^{2p}\middle| \mathcal{F}^{N}_{0}\right)  \leq& \left( 1-\frac{\kappa \Delta}{8}\right) \mathbb{E}\left( \left| X^{i,N}(t_{k})\right|^{2p}\middle| \mathcal{F}^{N}_{0}\right)+\left( C_1+C_2+C_3\right) \Delta\\
 		\leq& C_1^*\left(1+ \left| X^{i,N}(0)\right| 
 		^{2p}\right).
 	\end{align*}
 	We have for all $k\geq 0$ that
 	\begin{align*}
 		\mathbb{E}\left( \left| X^{i,N}(t_{k})\right|^{2p}\middle| \mathcal{F}^{N}_{0}\right)
 		\leq C^*\left(1+ \left| X^{i,N}(0)\right|	^{2p}\right).
 	\end{align*}
 	According to \cite{BAO2025} Lemma 4.1, from H{\"o}lder's inequality and the inequality $(a+b)^\theta\leq a^\theta+b^\theta$ for $a,b>0$ and $\theta\in(0,1]$, we can infer that for any $p\geq 2$ which is not an even integer,
 	\begin{align*}
 		\mathbb{E}\left( \left| X^{i,N}(t_{k})\right|^{p}\middle| \mathcal{F}^{N}_{0}\right)
 		\leq C^*\left(1+ \left| X^{i,N}(0)\right|	^{p}\right).
 	\end{align*}
 	Thus, the final estimate can be desired.
 \end{proof}
 
 \begin{theorem}\label{uniform-in time convergence}
 	Let  Assumptions \ref{g}, \ref{dos}, \ref{p}, \ref{GAMMA} and \ref{bound} hold with $2L_1-2L_5-8\lambda_{3}-3>0$ and $2(2q+1)\vee \frac{2q(2+\delta)}{\delta} \le p_0$, $\delta\in(0,1)$. Then, for any  $\Delta\in (0, \Delta^{**}] $, there exists a constant $C>0$, independent of $t$, $N$ and $\Delta$, such that 
 	\begin{align*}
 		\sup_{1 \leq i \leq N}\sup_{t\geq 0}\mathbb{E}\left|x^{i,N}(t)-X^{i,N}(t) \right|^2\leq C \Delta^{\frac{2}{2+\delta}},
 	\end{align*}
 	where $q$ is given in Assumption \ref{dos}, $p_0$ is given in Assumption \ref{p} and $\Delta^{**}$ is given in Lemma \ref{new}.
 \end{theorem}
 \begin{proof}
 	Denote $ e^{i}(t):=x^{i,N}(t)-X^{i,N}(t) $ for any $t\geq 0$. From \eqref{eq3.1} and It{\^o}'s formula,  we get for $\lambda>0$,
 	\begin{align*}
 		\mathbb{E}\left( 	e^{\lambda t}|e^{i}(t)|^{2}\right)
 		\leq&\mathbb{E}\bigg[ \int_{0}^{t} e^{\lambda s}\Big(\lambda |e^{i}(s)|^{2}+2\left\langle e^{i}(s),b(x^{i,N}(s),\mu_s^{x,N})-b_{\Delta}(X^{i,N}(\check{s}),\mu_{\check{s}}^{X,N}) \right\rangle \\
 		&\quad+\left| g(x^{i,N}(s),\mu_s^{x,N})-g(X^{i,N}(\check{s}),\mu_{\check{s}}^{X,N})\right|^2\\
 		&\quad+\int_{Z}\left| \gamma(x^{i,N}(s),\mu_s^{x,N},z)-\gamma(X^{i,N}(\check{s}),\mu_{\check{s}}^{X,N},z)\right| ^2 \nu(\mathrm{d}z)\Big) \mathrm{d}s\bigg],
 	\end{align*}
 	where by Remark \ref{remark} and the definition of $b_{\Delta}$,
 	\begin{align}\label{5.18}
 		&\left\langle e^{i}(s),b(x^{i,N}(s),\mu_s^{x,N})-b_{\Delta}(X^{i,N}(\check{s}),\mu_{\check{s}}^{X,N}) \right\rangle\notag\\
 		\leq &\left\langle e^{i}(s),b(x^{i,N}(s),\mu_s^{x,N})-b(X^{i,N}(s),\mu_s^{X,N}) \right\rangle+\left\langle e^{i}(s),b(X^{i,N}(s),\mu_s^{X,N}) -b(X^{i,N}(\check{s}),\mu_{\check{s}}^{X,N})\right\rangle\notag\\
 		&+\left\langle e^{i}(s),b(X^{i,N}(\check{s}),\mu_{\check{s}}^{X,N}) -b_{\Delta}(X^{i,N}(\check{s}),\mu_{\check{s}}^{X,N})\right\rangle\notag\\
 		\leq &-\left(L_1-\frac{1}{2}-\frac{1}{2}L_5 \right) \left|e^{i}(s) \right|^2+\frac{1}{2}L_5\mathbb{W}_2^2(\mu_s^{x,N},\mu_s^{X,N}) +\frac{1}{2}\left|e^{i}(s) \right|^2\notag\\
 		&+L_0(1+\left| X^{i,N}(s)\right|^{2q}+\left| X^{i,N}(\check{s})\right|^{2q} )\left| X^{i,N}(s)-X^{i,N}(\check{s}) \right|^2\notag\\
 		& +L_5\left(\left|X^{i,N}(s)-X^{i,N}(\check{s})\right|^2+ \mathbb{W}_2^2(\mu_s^{X,N},\mu_{\check{s}}^{X,N})\right)+\frac{1}{2}\left|e^{i}(s) \right|^2 \notag\\
 		&+\left( L_0\left(1+ \left| X^{i,N}(\check{s})\right|^{2q} \right)\left| X^{i,N}(\check{s})\right|^2+C\right)\left| X^{i,N}(\check{s})\right|^{2q} \Delta,
 	\end{align}
 	where $C=\left|b_1(0) \right|^2 $. Note that $\left| b(x,\mu)-b_\Delta(x,\mu)\right| \leq \left|b_1(x) \right|\left| x\right|^q \sqrt{\Delta}$. 
 	\par Employing Assumption \ref{GAMMA}, together with \eqref{5.18} and H{\"o}lder's inequality, we have
 	\begin{align*}
 		&\mathbb{E}\left( 	e^{\lambda t}|e^{i}(t)|^{2}\right)\\
 		\leq&\int_{0}^{t}e^{\lambda s}\Big[ \left( \lambda-2L_1+2L_5+8\lambda_{3}+3\right) \mathbb{E}\left|e^{i}(s) \right|^2+\left(4L_5+8\lambda_{3} \right)\mathbb{E} \left| X^{i,N}(s)-X^{i,N}(\check{s}) \right|^2\\
 		&\qquad+2L_0\left( \mathbb{E}(1+\left| X^{i,N}(s)\right|^{2q}+\left| X^{i,N}(\check{s})\right|^{2q} )^{\frac{2+\delta}{\delta}}\right) ^{\frac{\delta}{2+\delta}}\left( \mathbb{E}\left| X^{i,N}(s)-X^{i,N}(\check{s}) \right|^{2+\delta}\right)^{\frac{2}{2+\delta}}\\
 		&\qquad+ C\left(  \mathbb{E}\left| X^{i,N}(\check{s})\right|^{2}+\mathbb{E}\left| X^{i,N}(\check{s})\right|^{4q+2}+1\right) \Delta\Big]\mathrm{d}s.
 	\end{align*}
 	Following the proof of Lemma \ref{3.2}, it can also be shown that for any $p\geq 2$, 
 	\begin{eqnarray*}
 		\sup_{1 \leq i \leq N}\sup_{s \geq 0}\mathbb{E}\left| X^{i,N}(s)-X^{i,N}(\check{s})\right| ^{p} \leq C\Delta,
 	\end{eqnarray*}
 	and with the help of Lemma \ref{new}, choose $\lambda=2L_1-2L_5-8\lambda_{3}-3>0$, the desired assertion can be obtained.
 \end{proof}
 Based on Theorems \ref{4.2}, \ref{4.5}, \ref{uniform-in time PoC} and \ref{uniform-in time convergence}, we now estimate the distance between the numerical invariant measure and the exact invariant measure.
 
 \begin{theorem}	\label{4.6}
 	Let Assumptions \ref{g}, \ref{dos},\ref{p}, \ref{GAMMA}  and \ref{bound} hold with $(L_1\land 2L_3) -1-6L_5-2(m+1)\lambda_3>0$, $2L_1-2L_5-8\lambda_{3}-3>0$, $p_0>4$
 	and $2(2q+1)\vee \frac{2q(2+\delta)}{\delta} \le p_0$, $\delta\in(0,1)$.  Then, for any 
 	$\Delta\in (0,\Delta^{*}\land \Delta^{**}] $, there exists a constant  $C>0$, independent of $t,N,\Delta$,
 	\begin{align*}
 		\mathbb{W}^2_{2}(\Pi^{\Delta,N},(\mu^*)^{\otimes N})\leq CN \left( \Delta^{\frac{2}{2+\delta}}+\varphi(N)\right).
 	\end{align*}
 	where $\Delta^{*} $ and $\Delta^{**} $ are given in Theorem \ref{4.5} and Lemma \ref{new}, respectively. Moreover, 
 	\begin{equation*}
 		\mathbb{W}^2_{2}(\mu^{\Delta,N},\mu^*)\leq C \left( \Delta^{\frac{2}{2+\delta}}+\varphi(N)\right).
 	\end{equation*}
 \end{theorem}
 \begin{proof}
 	Since $\mu_t^{x^i}=\mu_t^x$ for every $i=1,2,\cdots,N $, let $\mu^{\otimes N}$ be the initial distribution of solutions $\left( x^1(t),\cdots,x^N(t)\right) $ and $\left( X^{1,N}(t),\cdots,X^{N,N}(t)\right) $, applying the triangle  inequality, estimate \eqref{Pt} and Corollary \ref{3.4}, we obtain for any 
 	$\mu\in\mathcal{P}_{2}\left(\mathbb{R}^d\right)$, correspondingly $\mu^{\otimes N}\in\mathcal{P}_{2}\left( (\mathbb{R}^d)^N\right)$,
 	\begin{align*}
 		\mathbb{W}^2_{2}(\Pi^{\Delta,N},(\mu^*)^{\otimes N})
 		\leq&\lim_{t\to\infty} 3\mathbb{W}^2_{2}(\mathbf{P}^{\Delta,N}_{t}\mu^{\otimes N},\mathbf{P}_{t}\mu^{\otimes N})\\
 		\leq&3 \lim_{t\to\infty}\mathbb{E}\left( \sum_{i=1}^{N}|x^i(t)-X^{i,N}(t) |^2\right)\\ 
 		\leq &6N\left( 
 		\limsup_{t\to\infty}
 		\mathbb E\left|x^i(t)-x^{i,N}(t)\right|^2
 		+
 		\limsup_{t\to\infty}
 		\mathbb E\left|x^{i,N}(t)-X^{i,N}(t)\right|^2
 	\right) ,
 	\end{align*}
 	combining with Theorem \ref{uniform-in time PoC} and Theorem \ref{uniform-in time convergence}, we get the first result. At last, along with \eqref{key},	the proof is complete.
 \end{proof}

	\section{Numerical experiments}
In this section, we present three examples to illustrate our work. Since the exact solution is unavailable for the concerned system, the convergence rate for the example is determined with respect to a proxy solution which is derived from an approximation at a smaller timestep or a larger number of particles. 

Below, \( \mathcal{N}(x, y) \) denotes the normal distribution with mean \( x \in \mathbb{R} \) and variance \( y \in (0, \infty) \), \( U(x, y) \) denotes the uniform distribution over \([x, y]\) for \(-\infty < x < y < \infty\), \( \lambda(x, y) \) denotes the Laplace distribution with location parameter \( x \in \mathbb{R} \) and scale parameter \( y > 0 \).
\begin{example}\label{e1}
	Consider the following MV-SDE with L{\'e}vy noise
	\begin{align*}
		\mathrm{d}x(t) = &\big(-4x(t) - 4x^3(t) + 0.05\,\mathbb{E}[x(t)]\big)\,\mathrm{d}t \notag\\
		&+ \big(1.5 + 0.4\,x(t)\big)\,\mathrm{d}B(t) 
		+ \int_{\mathbb{R}\backslash\{0\}} \sin\!\big(x(t)\big)\,z\;\widetilde{N}(\mathrm{d}t,\mathrm{d}z).
	\end{align*}
	It can be verified that 
	$b(x,\mu) = -4x - 4x^3 + 0.05\int_{\mathbb{R}}y\,\mu(\mathrm{d}y)$,
	$g(x,\mu) = 1.5 + 0.4x$, and 
	$\gamma(x,z) = \sin(x)\,z$
	satisfy Assumptions \ref{g} and \ref{dos}.
	And it follows from Theorem  \ref{4.2} and Remark \ref{remark} that the MV-SDE admits a unique invariant
	measure since	$ L_1-L_5-4\lambda_3-1 =0.3575>0$, where
	$L_1=2$, $L_5=0.0025$ and $\lambda_3=0.16$. Moreover, by virtue of Theorem~\ref{4.5}, the tamed Euler scheme~\eqref{discrete} admits a unique
	numerical invariant measure, since $(L_1\land 2L_3)-1-6L_5-2(m+1)\lambda_3=0.345>0$,
	where $L_3=4$ and $m=1$.
	
	Next, we show our simulation results for the approximation of the
	invariant measure. Let the jump size $z$ follow $\mathcal{N}(0,0.04)$ and the jump intensity be $1$.  Figure 1 shows the evolution of the empirical densities of numerical solutions generated by the tamed Euler scheme starting from different 
	initial distributions:
	$\mathcal{N}(0,1)$, $U(-3,3)$ and $\lambda(0,1)$ with time step $\Delta=2^{-8}$ at times $t = 0, 0.1, 5, 10$. It can be observed that empirical densities generated from different initial distributions gradually approach the same distribution as time evolves,  which illustrates the ergodicity property.
	Figure
	2 (a) compares the empirical distribution functions of the numerical solutions with time step $\Delta=2^{-8}$ at $t =
	10$ for the different	initial distributions mentioned above.
	The curves are almost indistinguishable, which supports our theoretical results as well.
	Taking \(\mathcal{N}(0,1)\) as the initial distribution, Figure~2\,(b) compares the numerical solution obtained with \(\Delta=2^{-8}\) with a reference solution computed using the smaller step size \(\Delta=2^{-12}\)  at $t=10$.
	The similarity between the two distributions indicates that the numerical invariant distribution provides a good approximation to the theoretical one.
	\begin{figure}
		\centering
		
		\includegraphics[width=\linewidth]{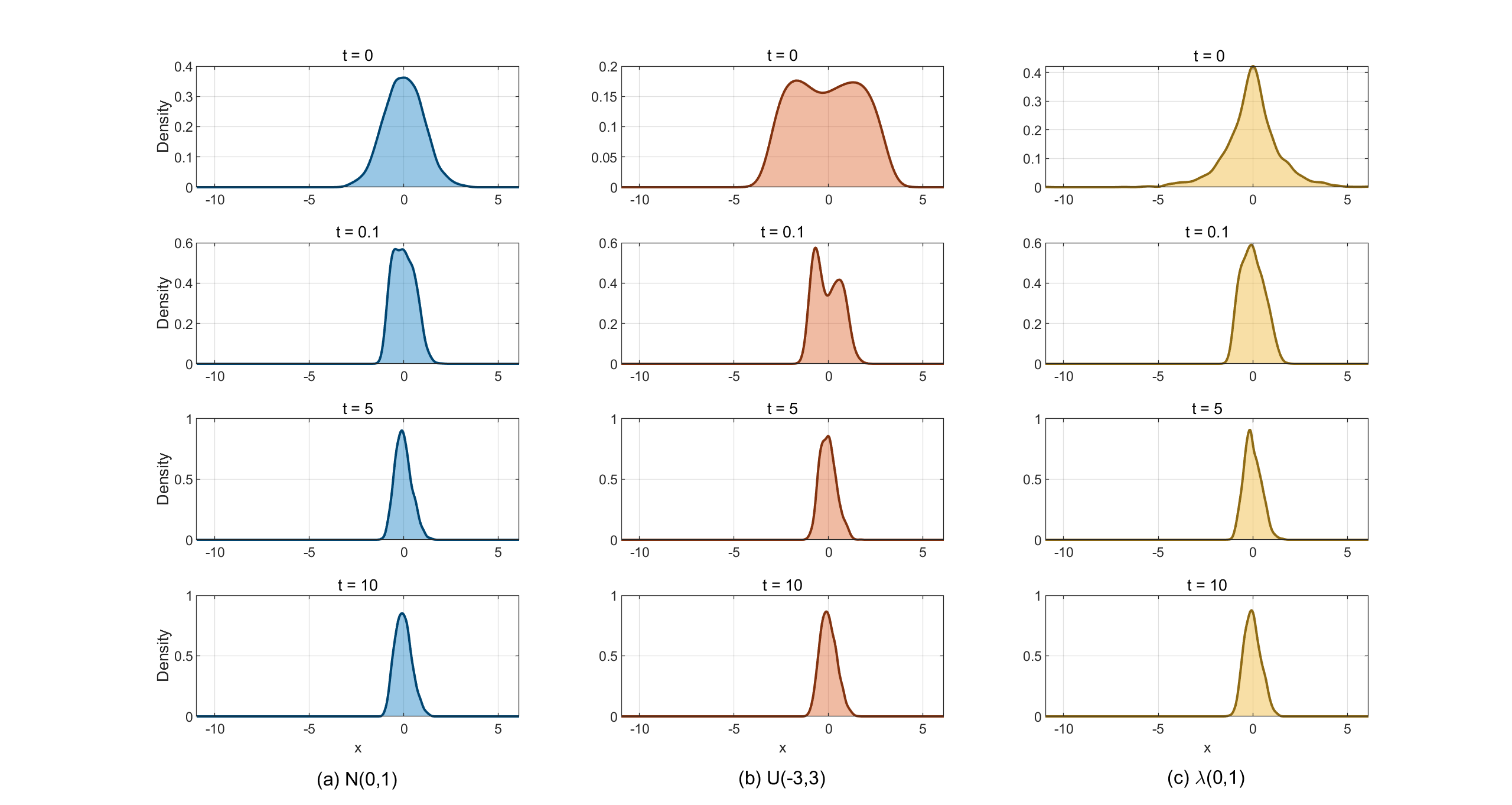} 
		\caption{Approximation of the invariant distribution of Example \ref{e1} with N = 2000 particles.}
		\label{fig:main}
	\end{figure}
	\begin{figure}
		\centering
		\begin{subfigure}{0.45\textwidth}
			\includegraphics[width=\linewidth]{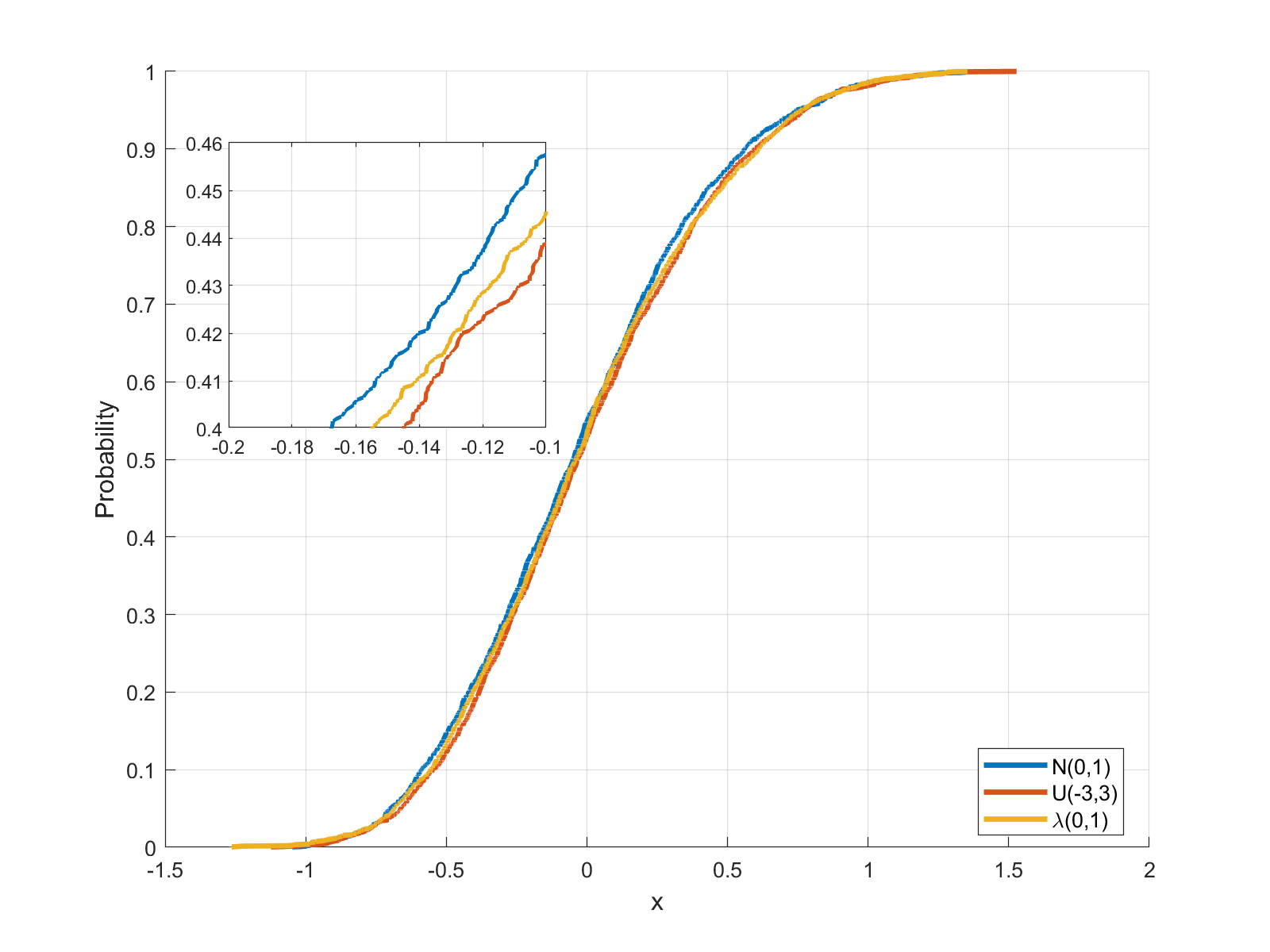} 
			\caption{The empirical distribution functions of the numerical solutions with different initial distributions}
			\label{fig:sub1}
		\end{subfigure}
		\hfill
		\begin{subfigure}{0.45\textwidth}
			\includegraphics[width=\linewidth]{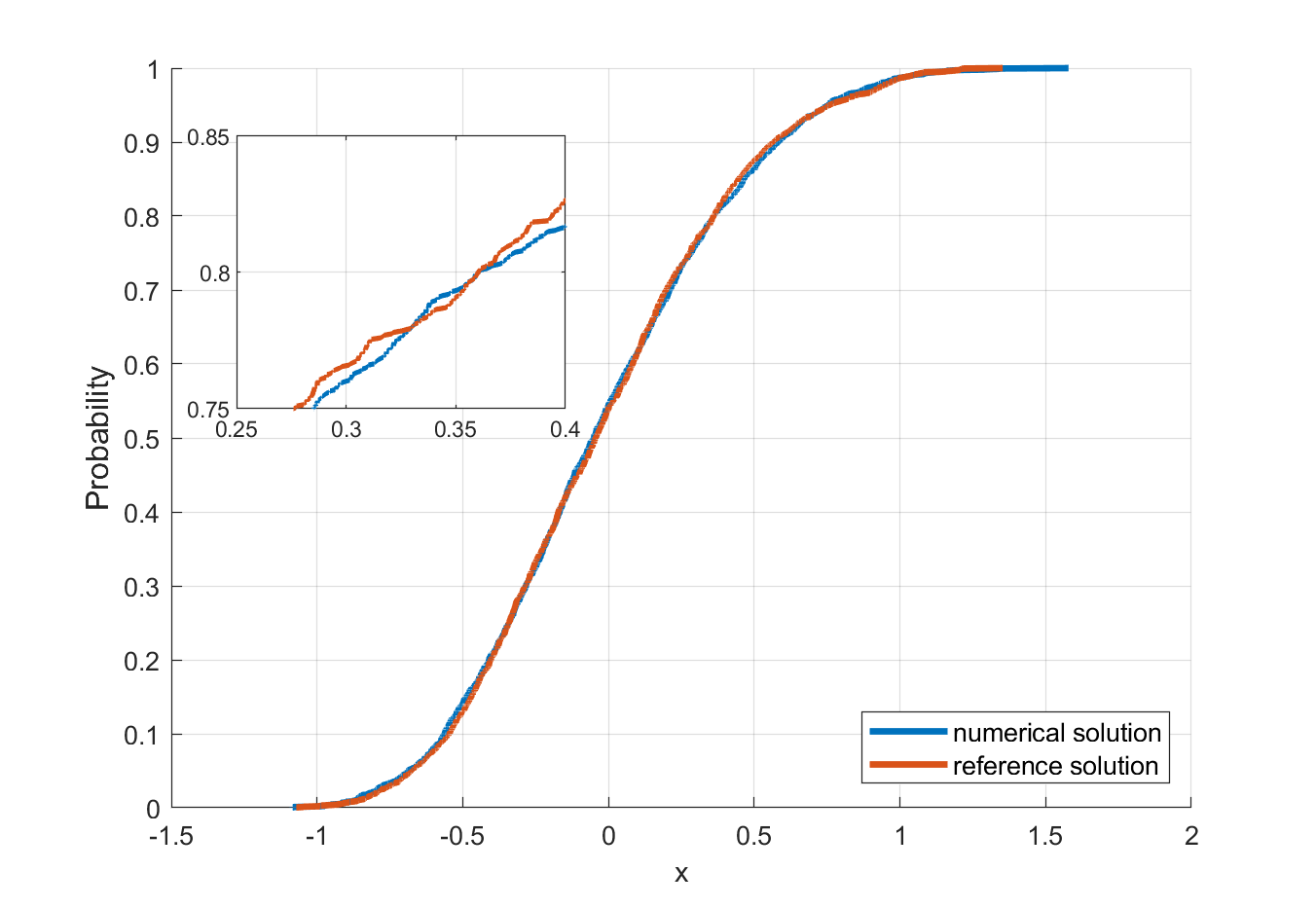} 
			\caption{The empirical distribution functions of the numerical solution and the reference solution
			}
			\label{fig:sub2}
		\end{subfigure}
		\caption{The empirical distribution functions for Example \ref{e1}  with N = 2000 particles.}
		\label{fig:main1}
	\end{figure}
\end{example}
\begin{example}\label{e2}
	Consider the following MV-SDE with L{\'e}vy noise
	\begin{align}
		\mathrm{d}x(t)=&
		\left(-2.5 x(t)-3.4x^3(t)+0.05\mathbb{E}(x(t))\right)\mathrm{d}t
		+\left(0.65+0.15\sin\left(\mathbb{E}(x(t))\right)\right)\mathrm{d}B(t) \notag\\
		&+\int_{\mathbb{R}\backslash\{0\}}\sin z\,\widetilde{N}(\mathrm{d}t,\mathrm{d}z).
		\notag
	\end{align}
	It can be verified that $b(x,\mu)=-2.5x-3.4x^3+0.05\int_{\mathbb{R}^d}y\mu(\mathrm{d}y),$ $g(x,\mu)=0.65+0.15\sin\left(\int_{\mathbb{R}^d}y\mu(\mathrm{d}y)\right),$ $\gamma(x,z)=\sin z$
	satisfy Assumptions \ref{g}, \ref{dos} and \ref{bound}. And it follows from Theorem  \ref{4.2} and Remark \ref{remark} that the MV-SDE admits a unique invariant
	measure since	$ L_1-L_5-4\lambda_3-1 =0.6075>0$, where
	$L_1=1.7$, $L_5=0.0025$ and $\lambda_3=0.0225$. Moreover, by virtue of Theorem~\ref{4.5}, the tamed Euler scheme~\eqref{discrete} admits a unique
	numerical invariant measure, since $(L_1\land 2L_3)-1-6L_5-2(m+1)\lambda_3=0.595>0$, where $ L_3=3.4$ and $m=1$. Furthermore, $2L_1-2L_5-8\lambda_3-3=0.215>0,$
	and hence Theorem \ref{4.6} is also satisfied.
	
	Let the jump size \(z\) follow \(\mathcal{N}(0,1)\) and let the jump intensity be \(1\).
	Figure 3 depicts the evolution of the empirical densities of numerical solutions  with time step \(\Delta=2^{-8}\) at times $t = 0, 0.1, 5, 10$  starting from the initial distributions \(\mathcal{N}(0,1)\), \(U(-3,3)\) and $\lambda(0,1)$.
	It can be observed that, as time evolves, all densities from different initial laws gradually approach the same distribution.
	
	 Since the convergence rates in Theorem \ref{4.6} are derived from the convergence rates of PoC and time discretization, we only test the $L_2$ error of PoC and time discretization to verify the result of Theorem \ref{4.6}. Fixing the step size $\Delta$ and choosing different number of particles $ N_k, k\in \mathbb{N} $, the error with respect to the number of particles $N$ is measured by
	  \begin{align}
	  Propagation~of~chaos~error (PoC-\!Error)\approx\frac{1}{N_k}\sum_{i=1}^{N_k}\left| X^{i,N}(T)-\!X^{i,N_k}(T)\right|^{2} ,\notag
	  \end{align}
	  where $ X^{i,N}(T) $ is the numerical approximation of $ x^{i,N}(T) $. The error  with respect to the step size $\Delta$ is quantified by 
	  \begin{align}
	  	Time~discretization~error (Time-\!Error)\approx\frac{1}{N}\sum_{i=1}^{N}\left| X_M^{i,N}(T)-\!X_{M_k}^{i,N}(T)\right|^{2},\notag
	  \end{align}
	  where  we fix the number of particles $ N $ and choose different choices of $ M_k\in \mathbb{N} $, then $ X_M^{i,N}(T) $ and $ X_{M_k}^{i,N}(T) $ mean the  numerical approximations of $ x^{i,N}(T) $ with different step sizes. 
	  
	 Figure~4 (left) displays time-step convergence test at $t=5$ with the time step $\Delta \in \{2^{-4},2^{-5},2^{-6},2^{-7},2^{-8}\}$ for fixed $N=5000$, where we regard the numerical solution with \(\Delta=2^{-10}\) as the reference solution. 
	Figure~4 (right) shows the particle-number convergence test at $t=5$ with particle numbers $N \in \{500,1000,1500,2000,2500\}$ for a fixed time step $\Delta=2^{-6}$, where the reference empirical distribution is computed with 
	$N=10000$ and $\Delta=2^{-9}$. The numerical results are consistent with the theoretical convergence rates in Theorem~\ref{4.6}.
\begin{figure}
	\centering
	
	\includegraphics[width=\linewidth]{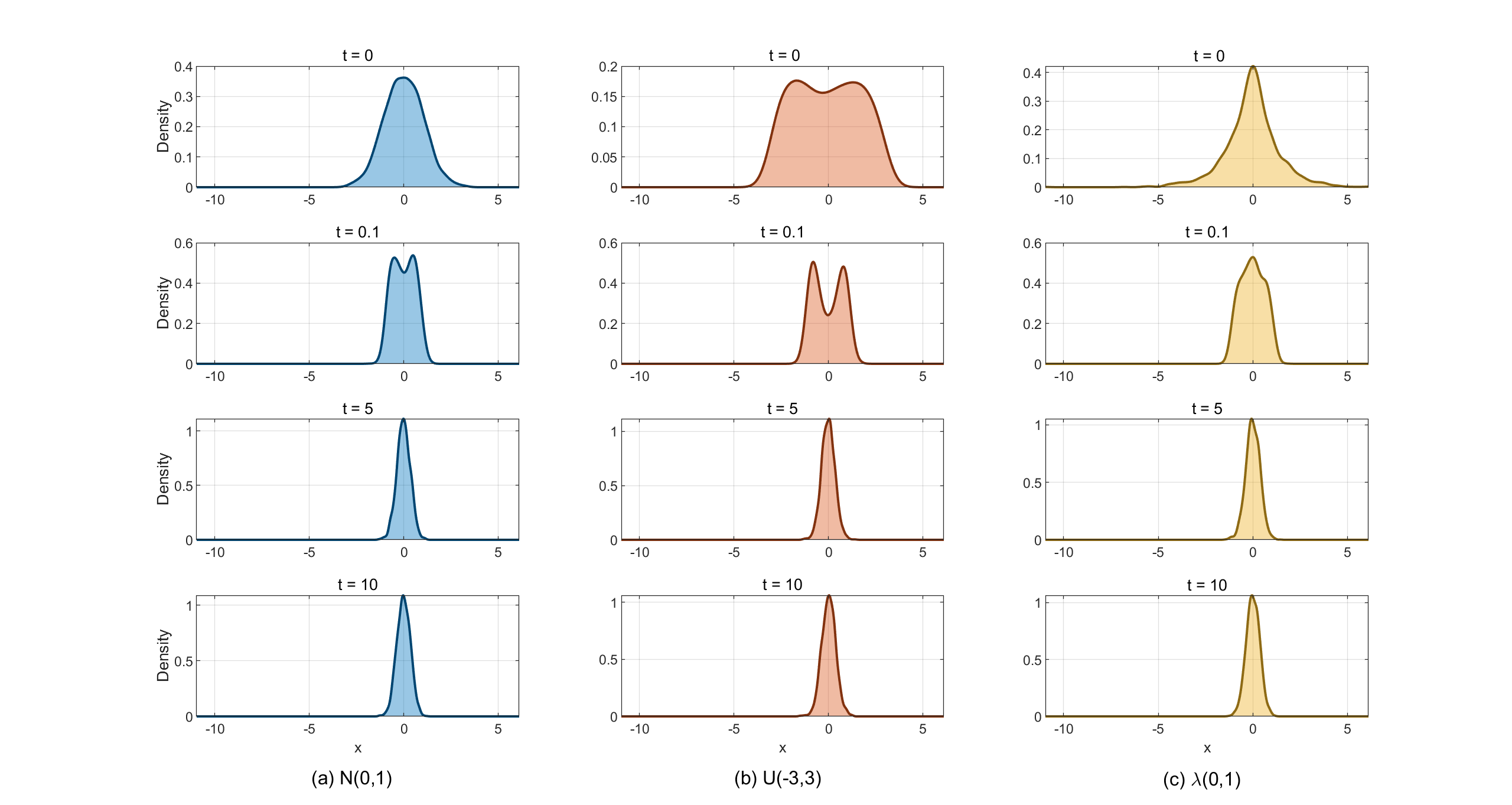} 
	\caption{Approximation of the invariant distribution of Example \ref{e2} with N = 2000 particles.}
	\label{fig:main2}
\end{figure}
\begin{figure}
	\centering
	\begin{subfigure}{0.45\textwidth}
		\includegraphics[width=\linewidth]{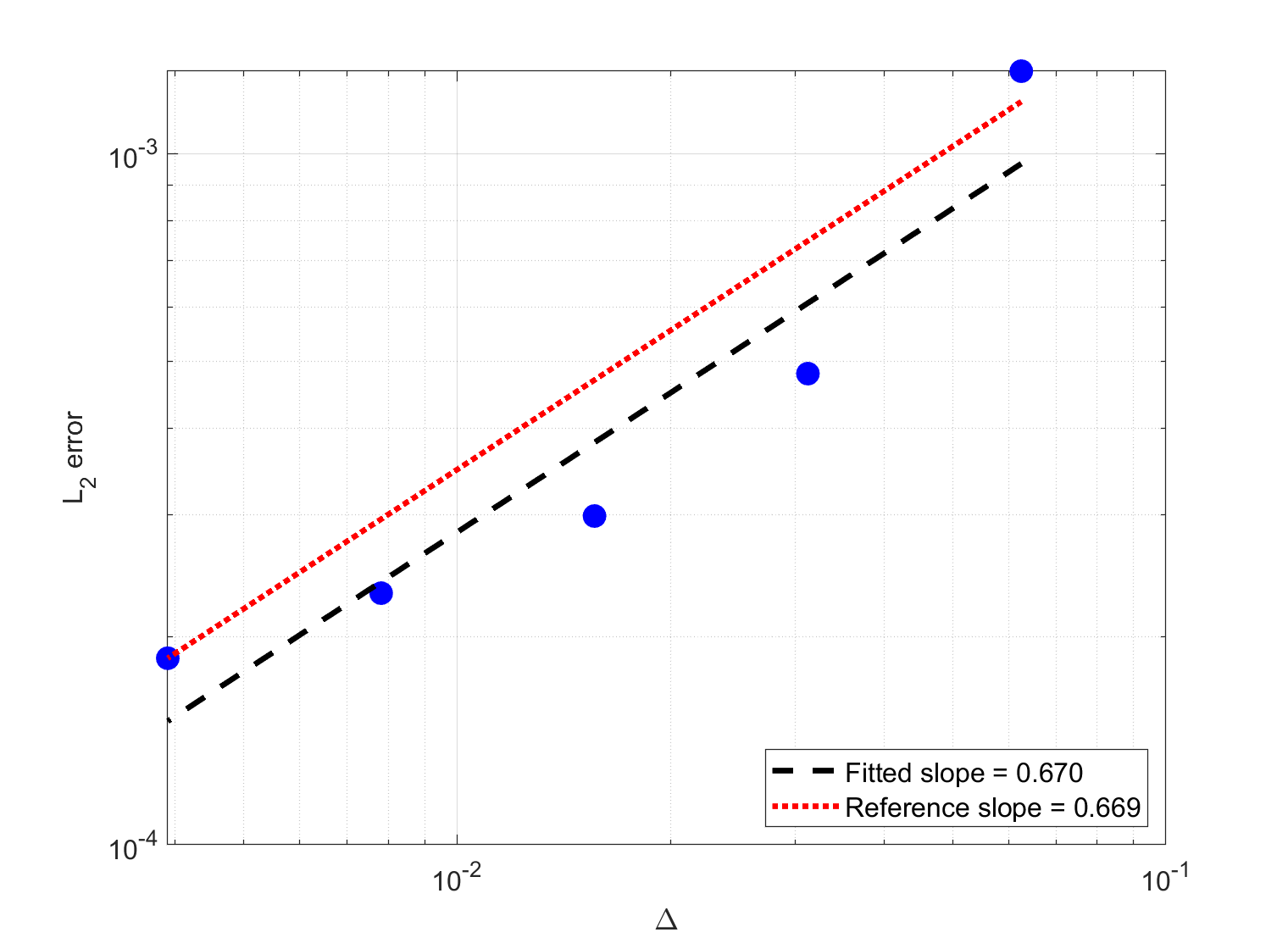} 
		\label{fig:sub3}
	\end{subfigure}
	\hfill
	\begin{subfigure}{0.45\textwidth}
		\includegraphics[width=\linewidth]{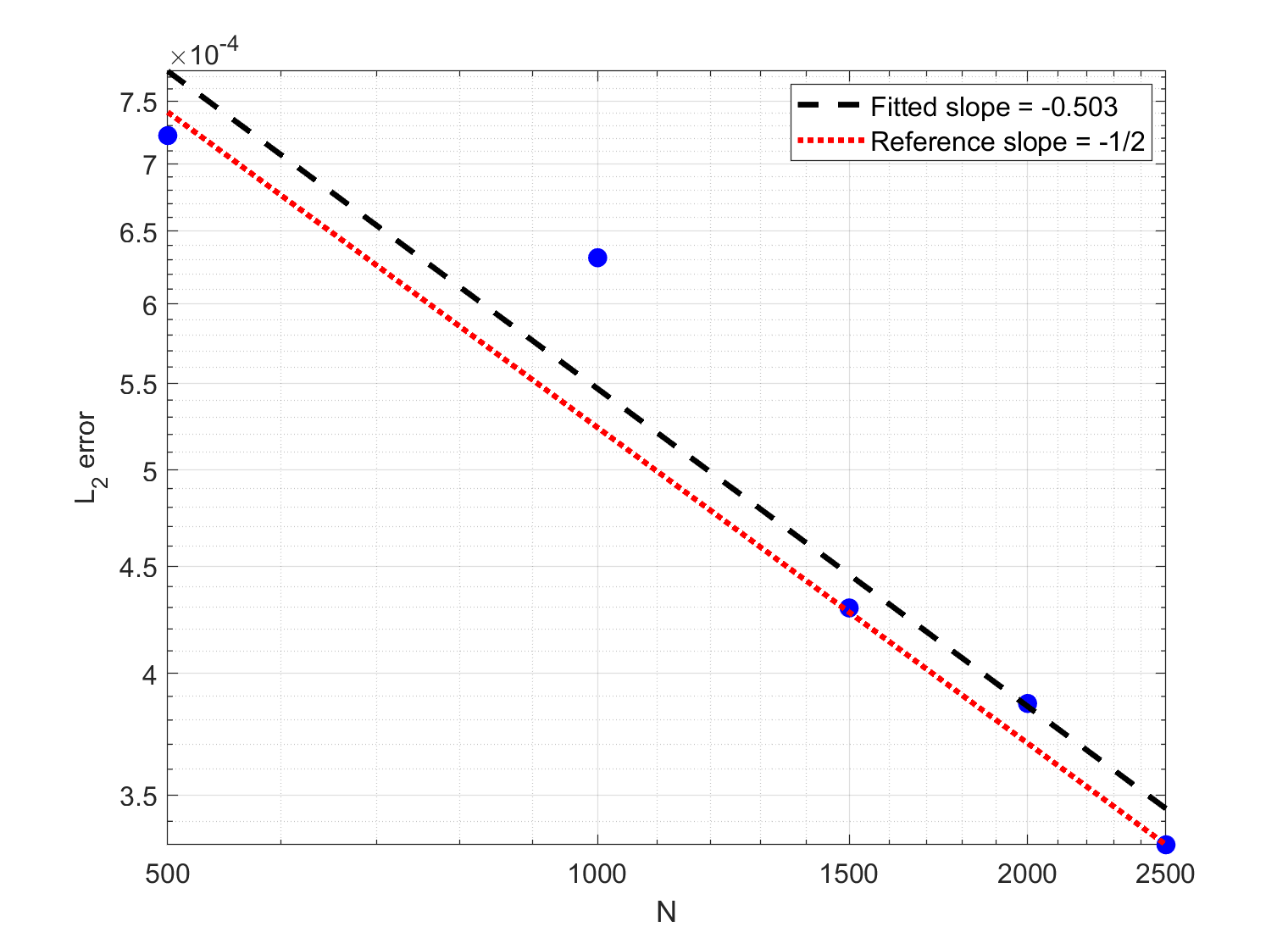} 
		\label{fig:sub4}
	\end{subfigure}
	\caption{Convergence rates of invariant measures for Example \ref{e2}  with $\delta=0.99$.}
	\label{fig:main3}
\end{figure}
\end{example}
 \begin{example}\label{e3}
 	Consider a 2-dimensional MV-SDE with L{\'e}vy noise
 		\begin{align}
 		\mathrm{d}x(t)=b(x(t),\mu_t^x) \mathrm{d}t+\mathrm{d}B(t)+\int_{Z} 	\begin{pmatrix}
 			\mathrm{sin} z\\
 			\mathrm{sin} z
 		\end{pmatrix}\widetilde{N}(\mathrm{d}t,\mathrm{d}z),\notag
 	\end{align}
 	where $x=(x_1,x_2)\in \mathbb{R}^2$ and
 	\[
 	b(x,\mu)=
 	\begin{pmatrix}
 		-1.8x_1+x_2-2.1x_1^3-2.1x_1x_2^2+0.05 \int_{\mathbb{R}^d}x_1\mu(\mathrm{d}x)\\
 		-1.8x_2-2.1x_2^3-2.1x_1^2x_2+0.05 \int_{\mathbb{R}^d}x_2\mu(\mathrm{d}x)
 	\end{pmatrix}
 	.
 	\]
 	A direct computation implies that $ b(x,\mu) $ satisfies Assumption \ref{dos}.  And it follows from Theorem  \ref{4.2} and Remark \ref{remark} that the MV-SDE admits a unique invariant
 	measure since	$ L_1-L_5-4\lambda_3-1 =0.0475>0$, where
 	$L_1=1.05$, $L_5=0.0025$ and $\lambda_3=0$. Moreover, by virtue of Theorem~\ref{4.5}, the tamed Euler scheme~\eqref{discrete} admits a unique
 	numerical invariant measure, since $(L_1\land 2L_3)-1-6L_5-2(m+1)\lambda_3=0.035>0$, where $ L_3=2.1$ and $m=2$.
 	
 	Let the jump size $ z $ follow \(\mathcal{N}(0,1)\) and the jump intensity be 1.
 	Figure~5 shows the evolution of the empirical densities of numerical solutions from three different initial distributions: $ \mathcal{N}(0,I_2),$ $ U(-3,3)^2,$ $ \lambda(0,1)^2$ with the step size $ \Delta=2^{-8} $ at times $t=0,0.1,2,10$. It can be observed that the  distributions become almost the same as time evolves.
 	
 	Taking \(\mathcal{N}(0,I_2)\) as the initial distribution, 
 	we compare the numerical solution computed with step size \(\Delta=2^{-8}\) with a reference solution computed with the smaller step size \(\Delta=2^{-12}\). 
 	Both solutions are simulated up to \(t=10\).
 	Figure~6 depicts the empirical density functions of the numerical solution and the reference solution in 3D and 2D settings,
 	respectively. 
 	The close agreement between the two empirical densities further illustrates the convergence of the numerical invariant measure obtained by the tamed Euler method.
 	\begin{figure}
 		\centering
 		
 		\includegraphics[width=\linewidth]{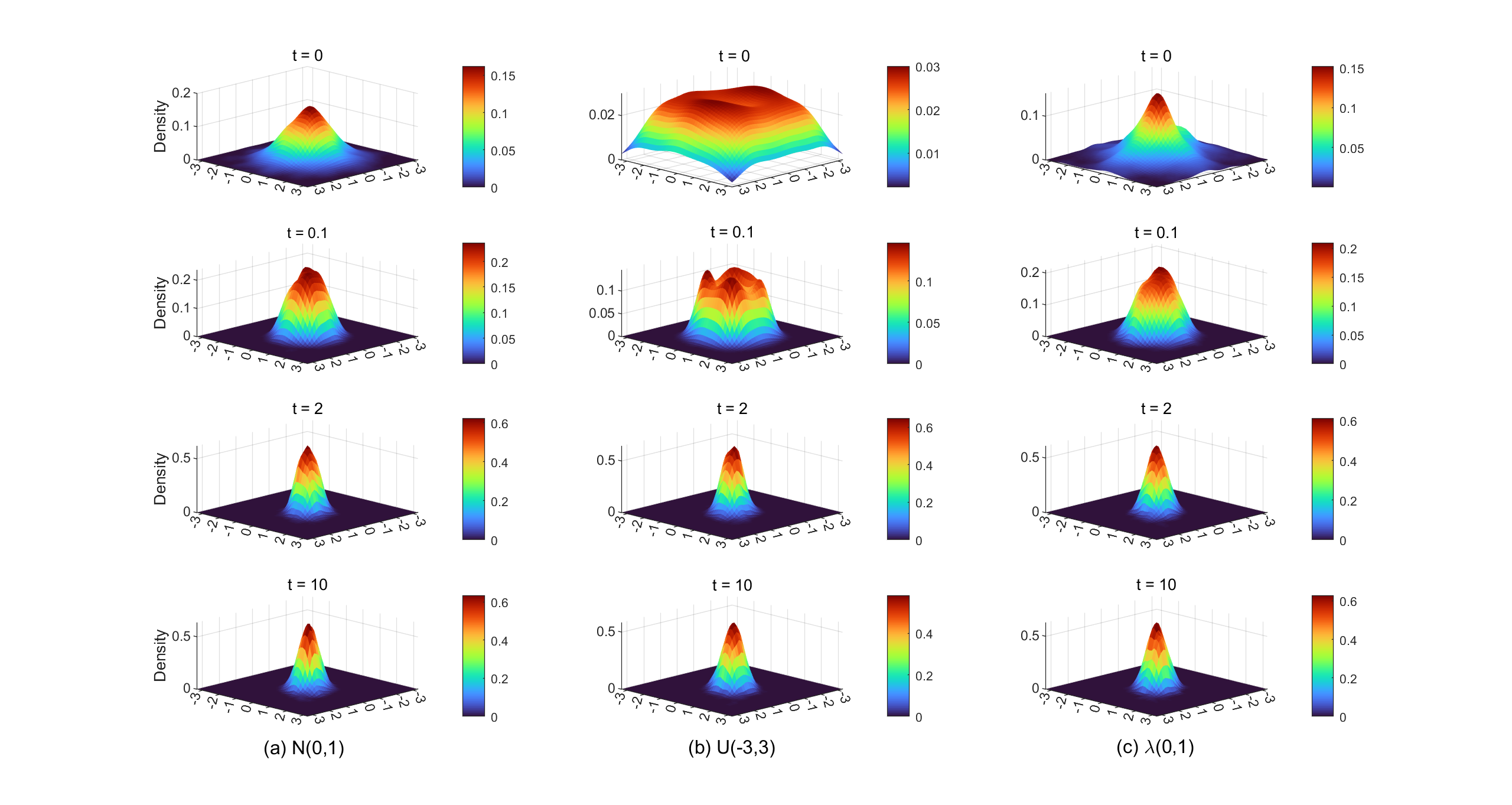} 
 		\caption{Approximation of the invariant distribution of Example \ref{e3} with N = 2000 particles.}
 		\label{fig:main4}
 	\end{figure}
 	\begin{figure}
 		\centering
 		
 		\includegraphics[width=\linewidth]{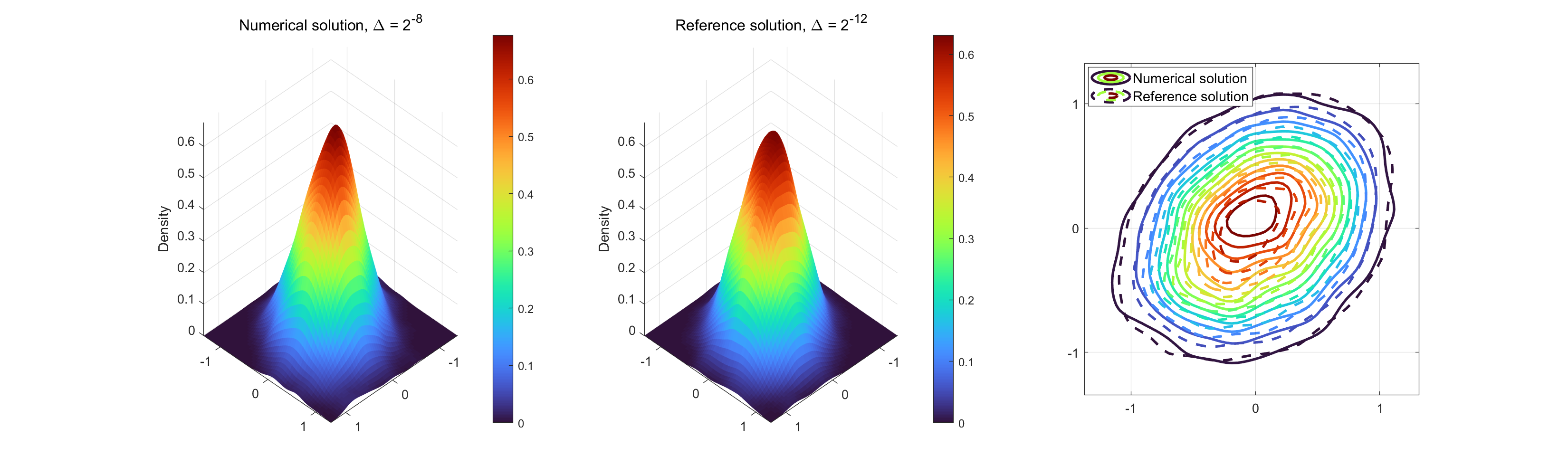} 
 			\caption{Comparison of empirical density functions  for Example \ref{e3} with N = 3000 particles.}
 		\label{fig:main5}
 	\end{figure}
 \end{example}
 \textbf{CRediT authorship contribution statement}
 \\
 
 \textbf{Yang Sun:} Writing – original draft; \textbf{Yuhang Zhang:}  Writing – review \& editing; \textbf{Minghui Song:} Supervision, Funding acquisition.
 \\
 
 \textbf{Data availability}
 \\
 
 No external data were used; all data in the numerical experiments were generated by simulations.
 \\
 
  \textbf{Declaration of competing interest}
  \\
  
  The authors declare that they have no known competing financial interests or personal relationships that could have appeared to influence the work reported in this paper.
  \\
  
  \textbf{Acknowledgment}
\\

  This work is supported by the National Natural Science Foundation of China(No. 12471372).



  \bibliographystyle{elsarticle-num-names} 
  \bibliography{myrefs}






\end{document}